\documentclass[11pt,reqno]{amsart}

\usepackage[T1]{fontenc} 
\usepackage[utf8]{inputenc}
\usepackage[english]{babel} 

\usepackage[dvipsnames]{xcolor}

\usepackage{url}
\usepackage{booktabs}
\usepackage{mathrsfs}
\usepackage{bbm}    % for \mathbbm{1}

\usepackage{amssymb}
\usepackage{amsmath} 
\usepackage{amsfonts,amssymb,mathtools}
\usepackage{amsthm} 
\usepackage{thmtools}

\usepackage[colorlinks=true, linkcolor=MidnightBlue, citecolor=OliveGreen, pagebackref]{hyperref}

\usepackage[capitalise, nameinlink, compress]{cleveref}
\usepackage[normalem]{ulem}

\newcommand{\PRG}{{\textup{PRG}}}

\DeclareMathOperator{\Z}{\mathbb{Z}}
\DeclareMathOperator{\Q}{\mathbb{Q}}
\DeclareMathOperator{\N}{\mathbb{N}}
\DeclareMathOperator{\R}{\mathbb{R}}
\DeclareMathOperator{\C}{\mathbb{C}}

\DeclareMathOperator{\id}{id}
\DeclareMathOperator{\Irr}{Irr}
\DeclareMathOperator{\Aut}{Aut}
\DeclareMathOperator{\rk}{rk}
\DeclareMathOperator{\absc}{absc}

\DeclareMathOperator{\Alt}{Alt}

\DeclareMathOperator{\SL}{\mathsf{SL}}

\DeclareMathOperator{\SU}{\mathsf{SU}}
\DeclareMathOperator{\Spin}{\mathsf{Spin}}
\DeclareMathOperator{\Sp}{\mathsf{Sp}}

\theoremstyle{plain}

\newtheorem{theorem}{Theorem}[section]
\newtheorem{lemma}[theorem]{Lemma}
\newtheorem{proposition}[theorem]{Proposition}
\newtheorem{corollary}[theorem]{Corollary}

\theoremstyle{definition}
\newtheorem{definition}[theorem]{Definition}
\newtheorem*{problem-IDP}{Invariant Degree Problem}
\newtheorem{example}[theorem]{Example}

\theoremstyle{remark}
\newtheorem{remark}[theorem]{Remark}

\numberwithin{equation}{section}
\title[Representations of quasi-semisimple groups]{Representation
  growth of \\ quasi-semisimple profinite groups}

\author[B. Klopsch]{Benjamin Klopsch} \address{Benjamin
  Klopsch: Heinrich-Heine-Universit\"at D\"usseldorf,
  Mathematisch-Na\-tur\-wis\-sen\-schaft\-liche Fakult\"at, Mathematisches
  Institut} \email{klopsch@math.uni-duesseldorf.de}
  
\author[M. Piccolo]{Margherita Piccolo} \address{Margherita Piccolo:
  FernUniversität in Hagen, Fakultät für Mathematik und Informatik}
\email{margherita.piccolo@fernuni-hagen.de}
  
\author[B. Sp\"ath]{Britta Sp\"ath} \address{Britta Sp\"ath: School of
  Mathematics and Natural Sciences, University of Wuppertal,
  Gau{\ss}str. 20, 42119 Wuppertal, Germany}
\email{bspaeth@uni-wuppertal.de}

\thanks{The research was conducted in the framework of the DFG-funded
  research training group ``GRK 2240: Algebro-Geometric Methods in
  Algebra, Arithmetic and Topology''.  The second author is a member
  of GNSAGA (INdAM) and kindly acknowledges their support.}
  


\begin{document}

\begin{abstract}
  The representation zeta function of a profinite group $G$ encodes
  the distribution of continuous irreducible complex representations
  of $G$ as a function of the dimension.  Its abscissa of convergence
  $\alpha(G)$ describes the polynomial degree of representation growth
  of~$G$.

  Within the class of quasi-semisimple profinite groups, we
  characterise those of polynomial representation growth (\PRG{}) and
  we prove that whether such a group~$G$ has \PRG{} or not only
  depends on its semisimple part $G/\mathrm{Z}(G)$.  Moreover, we show
  that, for quasi-semisimple profinite groups $G$ that have uniformly
  bounded Lie ranks, the degree of growth satisfies
  $\alpha(G) = \alpha(G/\mathrm{Z}(G))$.  We provide a technique to
  produce, for any prescribed positive real number~$\varrho$,
  quasi-semisimple profinite groups $G$ with \PRG{} of degree
  $\alpha(G) = \varrho$.  Our method allows for considerable
  flexibility regarding the inclusion of finite simple groups of Lie
  type as composition factors of~$G$.  Furthermore, we can arrange for
  the groups $G$ of prescribed representation growth to be profinite
  completions of suitable finitely generated discrete groups~$\Gamma$
  so that the group $\Gamma$ has the same representation zeta function
  as~$G$.
\end{abstract}

\maketitle

%%%%%%%%%%%%%%%%%%%%%%%%%%%%%%%%%%%%%%%%%%%%%%

\section{Introduction} \label{sec:intro}

Let $G$ be a profinite group.  For $n\in \N$, we denote by
$r_n(G)$ the number of isomorphism classes of $n$-dimensional
continuous irreducible complex representations of~$G$.  In the
situation that $G$ is \emph{representation rigid}, i.e., $r_n(G)$ is
finite for all $n \in \N$, the \emph{representation zeta
  function}
\[
  \zeta_G(s) = \sum\nolimits_{n=1}^\infty r_n(G) \, n^{-s} \qquad (s \in
  \C)
\]
is used to encode the representation growth of~$G$.  A key invariant
is the abscissa of convergence $\alpha(G)$ of $\zeta_G$. Unless $G$ is
finite (in which case we resort to other means), $\alpha(G)$ lies in
$\R_{>0} \cup \{ \infty \}$ and equals the polynomial degree
of representation growth of~$G$ in the following sense:
\[
  \alpha(G) = \limsup_{n \in \N} \frac{\log R_n(G)}{\log n}, \quad
  \text{where $R_n(G) = \sum\nolimits_{k=1}^n r_k(G)$ for $n \in \N$.}
\]
In particular, $G$ has \emph{polynomial representation growth}
(\PRG{}), i.e., $R_n(G)$ is polynomially bounded as a function of~$n$,
if and only if $\alpha(G) < \infty$.  For $G$ to be representation
rigid, it is necessary that $G$ is \emph{small}, i.e., that it has
only finitely many open subgroups of index $n$ for each positive
integer $n$; in particular, $G$ is necessarily countably based.

Considerable effort has been directed toward the study of
representation growth of arithmetic lattices in semisimple locally
compact groups and related open compact subgroups of adelic groups;
see \cite{AKOV16,AKOV16b,AiAv18} and the references therein.  For
example, consider $\mathsf{G} = \SL_d$, for some $d \ge 2$, and let
$S \subseteq \mathbb{P}$ be a set of primes.  We may take interest in
the representations of the profinite group
$H = \mathsf{G}_\infty(S) = \prod_{p \in \mathbb{P} \smallsetminus S}
\SL_d(\Z_p)$, which can be regarded as a `profinite form' of the
$\Q$-linear algebraic group~$\mathsf{G}$.  Clearly,~$H$ surjects onto
the profinite group
$G = \prod_{p \in \mathbb{P} \smallsetminus S} \SL_d(\mathbb{F}_p)$, a
cartesian product of quasi-simple finite groups, and thus the
representation growth of $G$ provides a lower bound for the
representation growth of~$H$.  In their work on arithmetic groups,
Avni, Klopsch, Onn, and Voll~\cite{AKOV16} briefly touched upon the
representation growth of quasi-semisimple profinite groups such
as~$G$.  A more systematic study in this direction was started by
Garc\'ia Rodr\'iguez~\cite[Chapter~6]{Gar16}, in joint work with
Klopsch.  Recently, Corob Cook, Kionke and Vannacci studied the Weil
zeta functions of certain quasi-semisimple profinite groups;
see \cite[Theorem~6.1]{CoKiVa24}.

Quasi-semisimple profinite groups, in particular those of `bounded
type', play a significant role in the subject of subgroup growth;
see~\cite[Chapter~10]{LS03}.  We say that a profinite group $G$ is
\emph{quasi-semisimple} if $G$ is perfect, i.e.,
$G = \overline{[G,G]}$, and
$G/\mathrm{Z}(G) \cong \prod_{i \in I} S_i$, where $S_i$, $i \in I$,
is a family of non-abelian finite simple groups.  Equivalently, a
profinite group $G$ is quasi-semisimple if it is the inverse limit of
finite quasi-semisimple groups, i.e., finite central products of
quasi-simple finite groups.  The latter are usually called
`semisimple' for short, but we prefer to use the term
`quasi-semisimple' to emphasise the possible presence of a non-trivial
centre, in analogy to the usage of `quasi-simple' versus `simple'.  In
this article, a \emph{semisimple} profinite group is a quasi-semisimple
profinite group with trivial centre.

Since the Schur multiplier of a finite direct product of finite
perfect groups is the direct product of their Schur multipliers,
quasi-semisimple profinite groups are precisely the factor groups of
cartesian products of quasi-simple groups.  In particular, every
quasi-semisimple profinite group~$G$ with \emph{semisimple part}
$G/\mathrm{Z}(G) \cong \prod_{i \in I} S_i$, as above, has a
\emph{universal covering group} $\widetilde{G}$ that is isomorphic to
the cartesian product $\prod_{i \in I} \widetilde{S}_i$, where
$\widetilde{S}_i$ denotes the universal covering group of the
non-abelian finite simple group~$S_i$ for each $i \in I$.

Our first aim is to provide answers to the basic questions: which
quasi-semisimple profinite groups are representation rigid and which
of them have \PRG{}?

\begin{definition} \label{def:Mn-mn}
  Let $G$ be an infinite, countably based, quasi-semisimple profinite
  group and write $G/\mathrm{Z}(G) \cong \prod_{i \in \N} S_i$
  as a cartesian product of non-abelian finite simple groups.  For $n
  \in \N$, we set 
  \[
    m_n(G) = \# \{i \in \N \mid R_n(S_i) > 1 \} \in \N_0 \cup
    \{\infty\}
  \]
  and observe that these multiplicity counts only depend on the
  semisimple part of~$G$.
\end{definition}

Building upon \cite[Theorem~6.1 and~6.3]{Gar16} and using results of
Liebeck and Shalev~\cite{LS05}, which rely on Deligne--Lusztig
theory, we obtain the following basic result, which is reminiscent of
analogous statements for the subgroup growth of profinite groups.

\begin{theorem} \label{thm:criteria-rigid-PRG} Let $G$ be an infinite,
  countably based, quasi-semisimple profinite group. \samepage
  \begin{enumerate}
  \item[\textup{(1)}] If $G$ is finitely generated then $G$ is
    representation rigid. \samepage
  \item[\textup{(2)}] If $G$ has \PRG{} then $G$ is finitely
    generated.
  \item[\textup{(3)}] The group $G$ has \PRG{} if and only if $m_n(G)$
    is polynomially bounded in $n$.
  \end{enumerate}
  Moreover, in \textup{(1)} and \textup{(2)} the converse does not
  generally hold.
\end{theorem}
  
In particular, whether a quasi-semisimple profinite group $G$ has \PRG{},
or not, depends solely on the semisimple part of~$G$.  Empirical
evidence suggests that we may even expect a positive answer to the
following question.

\begin{problem-IDP}
  Suppose that $G$ is a quasi-semisimple profinite group with \PRG{}.
  Does it follow that $G$ and its universal covering group
  $\widetilde{G}$ have the same degree of polynomial representation
  growth, i.e., does $\alpha(G) = \alpha(\widetilde{G})$ hold?
\end{problem-IDP}

Using ideas and results from~\cite{AKOV16}, we obtain a partial answer
to this question.  To state our theorem, it is convenient to define
the \emph{Lie rank} $\rk_{\mathrm{Lie}}(S)$ of a finite non-abelian
simple group~$S$.  Apart from finitely many cases arising from
exceptional isomorphisms~\cite[Theorem~37]{St68}, the rank of $S$ is
determined as follows.  For finite simple groups of Lie type, arising
as described in \cref{thm:tits} below, the Lie rank is simply the rank
of the corresponding simple algebraic group; the Suzuki and Ree groups
have rank~$2$; for alternating groups we
set~$\rk_{\mathrm{Lie}}(\Alt(n)) = n$; the $26$ sporadic groups are
pragmatically assigned the Lie rank~$1$.  In the exceptional cases,
one and the same group $S$ allows for more than one description and
its Lie rank is the minimum of the corresponding values.  For
instance, $\Alt(5) \cong \mathsf{PSL}(2,4) \cong \mathsf{PSL}(2,5)$
has Lie rank~$1$.

A quasi-semisimple profinite group $G$ with semisimple part
$G/\mathrm{Z}(G) \cong \prod_{i \in I} S_i$, as above, is said to have
\emph{uniformly bounded Lie ranks}, if
$\sup \{ \rk_{\mathrm{Lie}}(S_i) \mid i \in I \}$ is finite, and
\emph{unbounded Lie ranks} otherwise.  If
$\{ i \in I \mid \rk_{\mathrm{Lie}}(S_i) \le \ell \}$ is finite for
each $\ell \in \N$, we say that $G$ has \emph{growing Lie ranks}.  We
show that the Invariant Degree Problem has a positive answer for
quasi-semisimple groups that have uniformly bounded Lie ranks.

\begin{theorem} \label{thm:unif-bounded-ranks-invariant-alpha} Let $G$
  be a quasi-semisimple profinite group with \PRG{} and let
  $\widetilde{G}$ be its universal covering group.  If $G$ has
  uniformly bounded Lie ranks, then
  $\alpha(G) = \alpha(\widetilde{G})$.
\end{theorem}

Kassabov and Nikolov~\cite{KN06} established that for every positive
real number~$\varrho$, there exists a cartesian product $G$ of finite
alternating groups such that $G$ has \PRG{} and $\alpha(G) = \varrho$.
Moreover, they proved that any such group $G$ is a profinite
completion, in the sense that there exists a finitely generated
discrete group~$\Gamma$ such that~$G$ is isomorphic to the profinite
completion $\widehat{\Gamma}$.  We generalize their result and produce
a more flexible construction for quasi-semisimple profinite groups of
prescribed representation growth.  This includes, in particular,
semisimple profinite groups whose non-abelian composition factors are
finite simple groups of Lie type.

Our approach builds on pioneering work of Garc\'ia Rodr\'iguez and
Klopsch~\cite[Chapter~6]{Gar16} in this direction; compare
with~\cref{diff-to-Garcia-Klopsch}.  Unlike Garc\'ia Rodr\'iguez and
Klopsch, we can adjust the ranks of the finite simple groups involved
to construct finitely generated semisimple profinite groups which are
profinite completions.  For an explanation of the terms `Lie type' and
`quasi-semisimple profinite group in characteristic $p$ and of Lie
type $\lambda$' that are used in the statement of the next result, we
refer to \cref{def:lie-type-basic,def:lie-type-profinite}.

\begin{theorem} \label{thm:main-result} Let $\varrho \in \R_{>0}$ be a
  positive real number.

  \smallskip

  \noindent \textup{(1)} For every Lie type $\lambda = (\Phi,\tau)$
  such that $\varrho > \rk(\Phi) / \lvert \Phi^+ \rvert$ and every
  prime $p$ there exists a semisimple profinite group $H$ in
  characteristic $p$ and of Lie type $\lambda$ that has \PRG{} of
  degree $\alpha(H) = \varrho$.

  \smallskip
    
  \noindent \textup{(2)} Let $H_m$, $m\in \N$, be a sequence of
  semisimple profinite groups in positive characteristics $p_m$ and of
  Lie types $\lambda_m$ such that
  \[
    \rk(H_m) \to \infty \qquad \text{and} \qquad \alpha(H_m) \to
    \varrho \;\; \text{from below,} \qquad \text{as $m \to \infty$.}
  \]
  Then there is a semisimple profinite group~$G$, which has growing Lie
  ranks and whose composition factors originate from those of the
  groups~$H_m$, such that $G$ and its universal covering group
  $\widetilde{G}$ satisfy
  $\alpha(G) = \alpha(\widetilde{G}) = \varrho$.
\end{theorem}

We record two corollaries.  The first corollary follows from
\cref{thm:main-result}, upon applying a criterion of Kassabov and
Nikolov; see \cref{thm:profinite-completion} and
\cref{pro:factors-through-compl}.  The second corollary is a direct
consequence of \cref{thm:main-result} and provides further evidence
for an affirmative answer to the Invariant Degree Problem formulated
above.

\begin{corollary} \label{cor:prof-compl} For every
  $\varrho \in \R_{>0}$ there exist finitely generated discrete groups
  $\Gamma$ such that $\Gamma$ has \PRG{} with
  $\alpha(\Gamma) = \varrho$ and $\widehat{\Gamma}$ is a semisimple
  profinite group, whose composition factors are finite simple groups
  of Lie type; moreover, every finite dimensional complex
  representation of $\Gamma$ factors through the profinite
  completion~$\Gamma \hookrightarrow \widehat{\Gamma}$.
\end{corollary}

\begin{corollary} \label{cor:growing-rank-invariant-alpha} There are semisimple profinite groups $G$ that have growing Lie ranks and are
  such that $\alpha(G) = \alpha(\widetilde{G})$.
\end{corollary}

In view of the results of Kassabov and Nikolov for cartesian products
of alternating groups and \cref{thm:main-result} above, it is natural
to ask whether the spectrum of possible representation growth degrees
$\alpha(G)$ decreases, if one puts further structural restrictions on
the quasi-semisimple profinite groups~$G$.  Our final result shows
that imposing a bound on the minimal number of generators of $G$ does
not diminish the degree spectrum.

\begin{theorem} \label{thm:growing-ranks-2-gens} Let $G$ be a
  quasi-semisimple profinite group that has growing Lie ranks and
  \PRG{}.  Then there is an open subgroup $H \le_\mathrm{o} G$ such
  that every open normal subgroup $N \trianglelefteq_\mathrm{o} G$
  with $N \subseteq H$ is $2$-generated.

  However, for every $d \in \N$ there are $d$-generated
  quasi-semisimple profinite groups~$G$ that have uniformly bounded
  Lie ranks and \PRG{} such that every open subgroup
  $H \le_\mathrm{o} G$ requires at least $d$ generators.
\end{theorem}

A key ingredient of our work is \cref{thm:approximation}, which was
proved by Avni, Klopsch, Onn and Voll in~\cite{AKOV16}.  In an
appendix we prove \cref{thm:Approximation-simple-groups}, a variant of
their result that applies to finite simple groups of Lie type.  The
proof requires additional considerations; it also provides the
opportunity to supply some additional details and to streamline parts
of the argument given in~\cite{AKOV16}.

%%%

\medskip

\noindent \textit{Organisation.}  \cref{sec:finite} contains some
background material about finite groups of Lie type and quasi-simple
finite groups.  \cref{thm:criteria-rigid-PRG} and the main part of
\cref{thm:growing-ranks-2-gens} are established in \cref{sec:PRG}.
The remaining part of \cref{thm:growing-ranks-2-gens} and
\cref{thm:unif-bounded-ranks-invariant-alpha} are proved in
\cref{sec:Aproox}.  \cref{thm:main-result} and its
\cref{cor:prof-compl,cor:growing-rank-invariant-alpha} follow from
\cref{thm:for-every-a-we-find-G-that-has-PRG,thm:main-constr-for-given-rho}
which are established in \cref{sec:proofs}.  In
\cref{sec:AKOV-revisited} we obtain
\cref{thm:Approximation-simple-groups}, a variant of one of our key
tools, namely \cref{thm:approximation}.

\medskip

\noindent \textbf{Acknowledgement.} Some of the results in this
article form part of the second author's doctoral thesis~\cite{Pic24}. The second author wishes to thank Lucia Morotti and Lucas Ruhstorfer for insightful conversations.

%%%%%

\section{Finite groups of Lie type} \label{sec:finite}

For the background material collected in this section, we use
\cite{MT11} as a main reference; other useful accounts, with a view
toward Deligne--Lusztig theory, include \cite{DM20,GM20}.  Let~$\mathbb{F}$ be an algebraically closed field of positive
characteristic~$p$.  For any power $q$ of~$p$, let
${\mathbb{F}_q\subseteq \mathbb{F}}$ denote the finite subfield of~$q$
elements.  Let~$\mathbf{G}$ be a connected reductive
$\mathbb{F}$-algebraic group, and let
$F \colon \mathbf{G} \to \mathbf{G}$ be a Frobenius endomorphism
providing an $\mathbb{F}_q$-structure on~$\mathbf{G}$.  The fixed
point group $\mathbf{G}^F= \{g \in \mathbf{G}\mid F(g)=g\}$ is called
a \emph{finite group of Lie type}, and we refer to the rank of the
algebraic group~$\mathbf{G}$ as the \emph{Lie rank} of $\mathbf{G}^F$.
In some cases, there are multiple ways to realize a group as
$\mathbf{G}^F$ and strictly speaking the Lie rank depends on the
particular realization; compare with~\cite[Theorem~37]{St68}.

\begin{remark} \label{rem:exclude-suzuki-ree} By restricting to
  Frobenius endomorphisms (instead of considering more general
  Steinberg maps), we are excluding certain rank-$2$ groups known as
  Suzuki and Ree groups.  This is permissible as our main interest is
  in families of groups of growing Lie rank. For the same reason, it is often enough to focus on classical Lie types.
\end{remark}

\begin{remark}
  We take the opportunity to correct a related statement in
  \cite[Remark~3.6]{AKOV16}.  Partly due to a misprint in the first
  edition of~\cite{Car93}, the list of degrees of polynomials giving
  the unipotent character degrees of ${}^2 F_4(q^2)$ and $F_4(q)$ were
  not correctly reported in \cite[Remark~3.6]{AKOV16}.  In fact, the
  degree set for the twisted group ${}^2 F_4(q^2)$, viz.\
  $\{0,11,14,20,22,23,24\}$, is a subset of the degree set for the
  split group $F_4(q)$, viz.\ $\{0,11,14,15,20,21,22,23,24\}$, in line
  with other groups.
\end{remark}

A subgroup~$\mathbf{H} \le \mathbf{G}$ is said to be \emph{$F$-stable}
if $F(h)\in \mathbf{H}$ for all $h\in \mathbf{H}$.  Let $\mathbf{T}$
be an $F$-stable maximal torus in~$\mathbf{G}$.  Let $\Phi$ and
$\Phi^{\vee}$ denote the root and coroot system with respect to
$\mathbf{T}$.  As $\mathbf{T}$ is $F$-stable, the Frobenius
endomorphism acts naturally on the character group $X = X(\mathbf{T})$
and the cocharacter group $Y = Y(\mathbf{T})$.  Consider the euclidean
vector space~$X_{\R}= X\otimes_{\Z} \R$.  By \cite[Proposition
22.2]{MT11}, there exists an automorphism $\phi \in \Aut(X_{\R})$ of
finite order~$\delta$ and a $p$-power $q = q_F$ such that~$F$ acts as
$q\phi$ on~$X_{\R}$ and~$F^{\delta}|_X = q^{\delta} \id_X$.  In
particular, $F$ induces a graph automorphism, also denoted by~$\phi$,
on the Dynkin diagram $\Gamma_\Phi$ associated to $\Phi$.  Non-trivial
graph automorphisms of connected Dynkin diagrams are known to occur in
types $\mathsf{A}_\ell$ for $\ell \ge 2$, $\mathsf{D}_\ell$ for $\ell \ge 4$ and $\mathsf{E}_6$; see
\cite[Table~11.1]{MT11}.

%%%
\subsection{Finite simple groups of Lie type and their covering
  groups} \label{sec:classification}

Suppose that the algebraic group $\mathbf{G}$ is simple.  In this
case, the finite groups of Lie type arising as~$\mathbf{G}^F$ can be
classified in terms of the irreducible root system $\Phi$, the
parameters $q$ and~$\delta$ discussed above and also the isogeny type;
see \cite[Section~22.2]{MT11}.  For each $\Phi$, there are finitely
many different isogeny types of simple groups $\mathbf{G}$ with root
system~$\Phi$.  By \cite[Proposition 22.7]{MT11}, the Frobenius
endomorphisms of such groups $\mathbf{G}$ are induced by Frobenius
endomorphisms of the simply connected group $\mathbf{G}_\mathrm{sc}$.
In place of the parameter $\delta$ one may use the automorphism
$\tau = \phi |_\Phi$ of $\Phi$ stabilising a choice of positive roots
$\Phi^+$; compare \cite[Section~3.1]{AKOV16} and the references
therein.

Suppose additionally that $\mathbf{G}$ is of simply connected type.
Then the finite groups of Lie type arising as $\mathbf{G}^F$ are
parametrised by $\Phi$, $q$ and $\delta$ (or $\tau$).  Moreover, the
groups~$\mathbf{G}^F$ are quasi-simple, except in a few cases
described next; see~\cite[Theorem~24.17]{MT11}.

\begin{theorem}[Tits]
  \label{thm:tits}
  Let $\mathbf{G}$ be simply connected simple with Frobenius
  endomorphism~$F$, as above. Then, unless $\mathbf{G}^F$ is
  isomorphic to one of
  \[
    \SL_2(2), \quad \SL_2(3), \quad \mathsf{SU}_3(2),
    \quad \mathsf{Sp}_4(2), \quad \mathsf{G}_2(2),
  \]
  the group $\mathbf{G}^F$ is a finite quasi-simple group.  
\end{theorem}

We refer to the finite simple groups
$\mathbf{G}^F/ \mathrm{Z}(\mathbf{G}^F)$ arising in this way as
\emph{finite simple groups of Lie type}, tacitly ignoring the Suzuki
and Ree groups.  Where necessary, we deal with the Suzuki and Ree
groups separately; compare with \cref{rem:exclude-suzuki-ree}.

\begin{remark}\label{rem:covering-fin-exceptions}
  In the setting of \cref{thm:tits} and apart from finitely many fully
  documented exceptions, the finite group $\mathbf{G}^F$ is a
  universal covering group of the simple group
  $\mathbf{G}^F / \mathrm{Z}(\mathbf{G}^F)$ and
  $\mathrm{Z}(\mathbf{G}^F)$ is the Schur multiplier
  $M(\mathbf{G}^F / \mathrm{Z}(\mathbf{G}^F))$; see
  \cite[Remark~24.19]{MT11} and \cite[Tables~6.1.2
  and~6.1.3]{GLS97}.
  % \footnote{DELETE? -- For a simply connected simple
  %   algebraic group $\mathbf{G}$ the center $\mathrm{Z}(\mathbf{G}^F)$
  %   is a cyclic group, except for the case of root systems $D_n$ with
  %   $n \ge 4$ even and $\delta$ equal to $1$, where we have a product
  %   of two cyclic groups both of order $(2,q-1)$; see
  %   \cite[Corollary~24.13]{MT11}.}
\end{remark}

We also want to make use of the view taken in~\cite{AKOV16}, which
arises naturally from a global perspective taken in the context of
arithmetic groups.

\begin{definition} \label{def:lie-type-basic} We use the term
  \emph{Lie type} to refer to a pair $(\Phi,\tau)$, where $\Phi$ is a
  root system with an implicit choice of positive roots $\Phi^+$ and
  $\tau$ is an automorphism of $\Phi$ preserving~$\Phi^+$.  We say
  that a finite group of Lie type $\mathbf{G}^F$ has Lie type
  $\lambda = (\Phi_\lambda,\tau_\lambda)$ if the connected reductive
  algebraic group $\mathbf{G}$ has root system~$\Phi_\lambda$ and the
  action of the Frobenius endomorphism $F$ induces the same action
  as~$\tau_\lambda$ on the root system $\Phi$; compare with
  \cite[Chapter~4]{DM20} and \cite[Chapter~15]{Spr98}.

  Let $\mathbf{G}$ be a connected simply connected simple algebraic
  group in positive characteristic, equipped with a Frobenius
  endomorphism $F$ that defines an $\mathbb{F}_q$-structure on
  $\mathbf{G}$.  Put $\lambda = (\Phi,\tau)$, where $\Phi$ is the root
  system associated to~$\mathbf{G}$ and $\tau$ is the automorphism
  induced by $F$ on the root system~$\Phi$.  In this situation we
  write
  \[
    L_\lambda(q) = \mathbf{G}^F \qquad \text{and} \qquad S_\lambda(q)
    = L_\lambda(q) / \mathrm{Z}(L_\lambda(q)), 
  \]
  and we observe that, apart from finitely many exceptions,
  $L_\lambda(q)$ is a quasi-simple group of Lie type~$\lambda$ with
  simple part~$S_\lambda(q)$ and, moreover, $L_\lambda(q)$ is the
  universal covering group of~$S_\lambda(q )$; see \cref{thm:tits} and
  \cref{rem:covering-fin-exceptions}.
\end{definition}

%%%
\subsection{Generating quasi-simple groups}

The minimal number of generators of a group~$G$ is denoted by~$d(G)$.

\begin{proposition}
  \label{pro:center-fratt}
  Let $G$ be a finite perfect group.  Then the centre $\mathrm{Z}(G)$
  is contained in the Frattini subgroup~$\mathrm{Frat}(G)$.
\end{proposition}

\begin{proof}
  For a contradiction, suppose that $\mathrm{Z}(G)$ is not contained
  in~$\mathrm{Frat}(G)$.  Then there is a maximal
  subgroup~$M \le_\mathrm{max} G$ such that
  $\mathrm{Z}(G) \not\subseteq M$, and thus~$G = M \mathrm{Z}(G)$.
  This implies
  $G = [G,G] = [M\mathrm{Z}(G),M\mathrm{Z}(G)] = [M,M] \subseteq M$, a
  contradiction.
\end{proof}

The classification of finite simple groups (CFSG) implies that every
finite simple group is $2$-generated.  Using the previous proposition
we deduce the following.

\begin{corollary}
  \label{cor:quasi-simple-2-gen}
  Every quasi-simple finite group is $2$-generated.
\end{corollary}

Next, let~$G$ be a finite group, and let~$S_1, S_2,\dots, S_r$ denote,
up to isomorphism, the different non-abelian simple groups that occur
as factor groups of~$G$.  For $i \in \{1,\dots ,r\}$, let $m_i \in \N$
be maximal such that $G$ maps homomorphically onto $S_i^{\, m_i}$.
Results of Gasch\"utz and Wiegold show that, for $n \in\N$,
\begin{equation}
  \label{equ:number-of-gemn}
  d(G^n) = \max \bigl\{ d(G), n\, d(G/[G,G]), d(S_1^{\, m_1 n}), \dots
  , d(S_r^{\, m_r n}) \bigr\};
\end{equation}
see~\cite{Wie80}.  We are interested in cartesian powers of
quasi-simple finite groups.  For a quasi-simple finite group $G$, the
minimal number of generators of~$G^{|G|^b}$, for $b\in\N$, is easily
derived from~\cref{cor:quasi-simple-2-gen},
\eqref{equ:number-of-gemn}, and the theorem in~\cite{Wie80}.

\begin{theorem}[Wiegold] \label{thm:Wiegold-bound} Every quasi-simple
  finite group $G$ satisfies~$d(G^{|G|^b}) = b+2$ for all~$b\in \N$.
\end{theorem}

%%%%%

\section{Rigidity and polynomial representation growth}
\label{sec:PRG}

Let $G$ be an infinite, countably based, quasi-simple profinite group.
In this section we establish \cref{thm:criteria-rigid-PRG}, which
addresses the interconnection between $G$ being finitely generated,
representation rigid or of polynomial representation growth (\PRG{}).
Taking the analysis somewhat further, we also obtain a proof of the
main part of \cref{thm:growing-ranks-2-gens}.

We recall from \cref{def:Mn-mn} that, for
$G/\mathrm{Z}(G) \cong \prod_{i \in \N} S_i$ with non-abelian finite
simple composition factors~$S_i$ and $n \in \N$, we write
\[
  m_n(G) = \# \{i \in \N \mid R_n(S_i ) >1 \}.
\]
The universal covering group $\widetilde{G}$ of $G$ satisfies
$\widetilde{G} \cong \prod_{i \in \N} \widetilde{S}_i$, where
$\widetilde{S}_i$ denotes the universal covering group of $S_i$ for
each $i \in \N$.  For $n \in \N$, we set
\[
  \widetilde{m}_n(G) = \# \{i \in \N \mid R_n(\widetilde{S}_i
  ) >1 \}.
\]

\begin{proposition} \label{pro:m(n)-polynomial-iff-cover} Let $S$ be a
  non-abelian finite simple group with universal covering
  group~$\widetilde{S}$.  Let $\chi$ be a non-principal irreducible
  character of~$\widetilde{S}$.  Then $S$ has a non-principal irreducible character of
  degree at most $\chi(1)^2-1$.
\end{proposition}

\begin{proof}
  Set $Z = \mathrm{Z}(\widetilde{S})$ and write $\chi_Z$ for the
  restriction of $\chi$ to~$Z$.  As $Z$ is central
    in~$\widetilde{S}$, the character $\chi_Z$ has a unique
  irreducible constituent $\nu \in \mathrm{Irr}(Z)$, and
  $\chi(z) = \chi(1) \nu(z)$ for $z\in Z$.

  The character $\xi = \chi \overline{\chi}$ is the inflation of a
  character of~$S$, because
  \[
    \xi(z) = \chi(1)^2 |\nu(z)|^2 = \chi(1)^2 = \xi(1) \qquad \text{for all
      $z\in Z$}
  \]
  and thus $Z \subseteq \mathrm{ker}(\chi)$.  Moreover, we see that
  $\xi \ne \xi(1)\, \mathbbm{1}_{\widetilde{S}}$, where
  $\mathbbm{1}_{\widetilde{S}}$ denotes the principal character; for
  otherwise
  $\lvert \chi(g) \rvert^2 = \lvert \xi(g) \rvert = \xi(1) =
  \chi(1)^2$ would give $\lvert \chi(g) \rvert = \chi(1)$ for all
  $g \in \widetilde{S}$, thus $\chi(1) = 1$ and
  $\widetilde{S}/\mathrm{ker}(\chi)$ would constitute a non-trivial
  abelian quotient of the perfect group $\widetilde{S}$.  Finally, we
  observe that $\mathbbm{1}_{\widetilde{S}}$ is a constituent
  of~$\xi$, as can be seen from the inner product calculation
  $\langle \xi, \mathbbm{1}_{\widetilde{S}} \rangle = \langle \chi
  \overline{\chi}, \mathbbm{1}_{\widetilde{S}} \rangle = \langle\chi, \chi
  \rangle = 1$.

  Thus $\xi$ admits a non-principal irreducible constituent, yielding
  a non-principal irreducible character of $S$ of degree at most
  $\xi(1)-1 = \chi(1)^2 -1$
\end{proof}

\begin{corollary} \label{cor:m(n)-polynomial-iff-cover} Let $G$ be an
  infinite, countably based, quasi-semisimple profinite group.  Then
  $m_n(G)$ is polynomially bounded as a function of $n$
  if and only if $\widetilde{m}_n(G)$ is polynomially bounded as a function of $n$.
\end{corollary}

\begin{proof}
  Let $n \in \N$.  Clearly, $m_n(G) \le m_n(\widetilde{G})$ and,
  conversely, \cref{pro:m(n)-polynomial-iff-cover} shows that
  $m_{n^2}(G) \ge m_n(\widetilde{G})$, where $\widetilde{G}$ denotes
  the universal covering group of~$G$.
\end{proof}

We continue with a consequence of Jordan's theorem on
finite linear groups.

\begin{proposition} \label{pro:characterise-rigid} Let $G$ be a
  quasi-semisimple profinite group.  Then $G$ is representation rigid
  if and only if the multiplicity of each non-abelian simple
  composition factor in the semisimple quotient $G/\mathrm{Z}(G)$ is
  finite.  In particular, if $G$ is finitely generated, then $G$ is
  representation rigid.
\end{proposition}

\begin{proof}
  If the multiplicity of a non-abelian simple composition factor $S$
  in $G/\mathrm{Z}(G)$ is not finite, then $G$ maps onto
  $\prod_{i \in \N} S$ and thus cannot be representation rigid.

  Now suppose that the multiplicity of each non-abelian simple
  composition factor in $G/\mathrm{Z}(G)$ is finite.  Every finite
  group is representation rigid.  Thus we may suppose that $G$ is
  infinite and
  $G/ \mathrm{Z}(G) \cong \prod_{i \in \N} S_i$ with
  non-abelian finite simple composition factors~$S_i$.  We fix
  $n \in \N$.  By Jordan's theorem on finite linear groups, there
  exists an upper bound $b(n)$ such that every finite subgroup of
  $\mathsf{GL}_n(\C)$ has an abelian normal subgroup of index at
  most~$b(n)$.  This implies that
  $r_n(G) \le r_n(\widetilde{G}_I) < \infty$, where
  \[
    \widetilde{G}_I = \prod\nolimits_{i \in I} \widetilde{S_i}, \qquad
    \text{for $I = \{ i \in \N \mid \lvert S_i \rvert \le b(n) \}$
    with $\lvert I \rvert < \infty$,}
  \]
  is the universal covering group of the finite semisimple group
  $\prod_{i \in I} S_i$.
\end{proof}

\begin{example}
  Let $S_i$, $i \in \N$, be any sequence of pairwise non-isomorphic
  non-abelian finite simple groups.  We pick any function
  $f \colon \N \to \N$ that diverges to infinity, and, for
  $i \in \N$, we set $n_i = \lvert S_i \rvert^{f(i)}$.  Then
  $G = \prod_{i \in I} S_i^{\, n_i}$ is a representation rigid
  semisimple group that is not finitely generated, due to
  \cref{thm:Wiegold-bound} and \cref{pro:characterise-rigid}.
\end{example}

Next we slightly extend a result of Liebeck and Shalev, who
established several asymptotic results for the representation zeta
functions of finite quasi-simple groups in~\cite{LS04,LS05}.  We
recall that, for $m \ge 8$, the universal covering group of the
alternating groups $\Alt(m)$ is a double cover, which we denote by
$2.\!\Alt(m)$; compare with~\cite[Theorem~5.2.3]{GLS97}.  First we
record the following complement to \cite[Corollary~2.7]{LS04}.

\begin{proposition} \label{pro:cover-alt} For every positive real
  number $\sigma \in \R_{>0}$,
  \[
    \zeta_{2.\!\Alt(m)}(s) =1+O(m^{-\sigma}) \qquad \text{as
      $m \to \infty$.}
  \]
\end{proposition}

\begin{proof}
  The analysis of irreducible characters of $\Alt(m)$ in \cite{LS04}
  yields, for $\sigma \in \R_{>0}$,
  \[
    \sum_{\chi \in \Irr(\Alt(m))}\chi(1)^{-\sigma} = 1 +
    O(m^{-\sigma}) \qquad \text{as $m \to \infty$.}
  \]
  The irreducible representations of the covering group $2.\!\Alt(m)$
  that do not factor through $\Alt(m)$ are the irreducible spin
  representations, which correspond to projective representations of
  $\Alt(m)$ and were studied by Schur; compare with~\cite{Ste89}.

  Let $\Irr_\mathrm{sp}(2.\!\Alt(m))$ denote the set of irreducible
  spin characters and recall that these are parametrised by partitions
  of $m$ with distinct parts and a possible sign choice; see
  \cite[Corollary~7.5]{Ste89}.  Classical bounds for the number $p(m)$
  of partitions thus yield
  \[
    \lvert \Irr_\mathrm{sp}(2.\!\Alt(m)) \rvert \le 2\cdot p(m)\le
    2\cdot \exp \bigl( \pi \sqrt{2m/3} \bigr) \le 14^{\sqrt{m}}
  \]
  for all sufficiently large~$m$; compare with
  \cite[Theorem~6.10]{Nar83}.  By \cite[Corollary~3.2]{KW91}, the
  degree of an irreducible spin representation of $2.\!\Alt(m)$ is at
  least~$2^{(m-3)/2}$.  These crude estimates yield, for
  $\sigma \in \R_{>0}$, the required bound
  \[
    \sum_{\chi \in \Irr_\mathrm{sp}(2.\!\Alt(m))}\chi(1)^{-\sigma} =
    O(2^{-m \sigma/3}) \qquad \text{as $m \to \infty$.} \qedhere
  \]
\end{proof}

Combining~\cref{pro:cover-alt} with~\cite[Corollary~1.4\,(ii)]{LS05},
which takes care of the classical groups, we obtain the following
consequence.

\begin{corollary}\label{cor:upper-bound-unbounded-rank}
  For every $\varepsilon \in \R_{>0}$ there exists $m \in \N$ such
  that, if $S$ is a quasi-simple group of Lie rank at least $m$, then 
  \[
    r_n(S) \le n^\varepsilon \quad\text{for all }n\in \N.
  \]
\end{corollary}

Extending the proof of \cite[Theorem~6.1]{Gar16}, we establish the
following result.

\begin{proposition} \label{thm:PRG-iff-m(n)-polynomial} Let $G$ be an
  infinite, countably based, quasi-semisimple profinite group.  Then
  $G$ has \PRG{} if and only if $m_n(G)$ is polynomially bounded
  in~$n$.
\end{proposition}

\begin{proof}
  By considering finite dimensional irreducible representations of $G$
  that factor through individual factors $S_i$ of
  $G/\mathrm{Z}(G) \cong \prod_{i \in \N} S_i$, we already
  obtain
  \[
    R_n(G)\ge R_n(G/\mathrm{Z}(G)) \ge m_n(G) \qquad \text{for all
      $n \in \N$,}
  \]
  and we conclude: if $G$ has \PRG{} then 
  $m_n(G)$ grows at most polynomially in~$n$.
  
  Conversely, suppose that $m_n(G)$ is polynomially bounded in $n$.
  Since $G$ is a quotient of its universal covering group
  $\widetilde{G} =\prod_{i \in \N}\widetilde{S_i}$, where
  $\widetilde{S_i}$ is a universal covering of $S_i$ for each
  $i \in \N$, it suffices to bound the representation growth of
  $\widetilde{G}$.  Moreover, using
  \cref{cor:m(n)-polynomial-iff-cover}, there exists $b \in \N$ such
  that $\widetilde{m}_n(G)\le n^b$ for all $n \in \N$.
  
  Let $n \in \N$.  Up to isomorphism, every $n$-dimensional
  irreducible representation~$\rho$ of~$\widetilde{G}$ factors through
  a finite quotient $\prod_{i \in I} \widetilde{S_i}$, for
  $I \subseteq \N$ with $\lvert I \rvert \le \log_2(n)$, and takes the
  form of a pullback of the tensor product
  $\boxtimes_{i \in I} \rho_i$ of non-trivial irreducible
  representations $\rho_i$ of $\widetilde{S_i}$ such that
  $d_i = \dim(\rho_i) \ge 2$, for each $i \in I$, and
  $n = \dim \rho = \prod_{i \in I} d_i$.  By~\cite[Lemma~4.7]{LM04},
  there is a constant $\mu \in \N$, independent of~$n$, such that the
  number of ordered factorizations of~$n$, i.e., finite tuples
  $\mathbf{d} = (d_j)_{j=1}^\ell$ with components in $\N_{\ge 2}$ such
  that $n = \prod_{j =1}^\ell d_j$, is bounded above by~$n^{\mu}$.
  Moreover, given any such factorization
  $\mathbf{d} = (d_j)_{j=1}^\ell$, for each
  $j \in \{1, \ldots, \ell \}$ the number of $i \in \N$ such that
  $\widetilde{S_i}$ has a representation of dimension~$d_j$ is at most
  $\widetilde{m}_{d_j}(G) \le d_j^{\, b}$.  Hence, the number of
  choices $(i_1,\dots,i_\ell) \in \N^\ell$ for matching indices, and
  hence for suitable subsets $I \subseteq \N$ of
  size~$\ell \le \log_2(n)$, is bounded
  by~$\prod_{j=1}^\ell d_j^{\, b} = n^b$.

  Thus it suffices to prove that $r_n(\widetilde{S})$ is polynomially
  bounded in~$n$, across all non-abelian finite simple groups~$S$.
  Clearly, it is always possible to accommodate for finitely many such
  groups with somewhat exceptional behaviour.  In particular, we do
  not need to discuss the sporadic groups, the Tits group, small
  Suzuki and Ree groups whose Schur multipliers do not vanish or
  simple groups of Lie type with non-generic Schur multipliers.  The
  following cases remain and can be dealt with as described:

  (a) Suzuki and Ree groups with trivial Schur multiplier.  The
  required polynomial bound can be established by using
  Deligne--Lusztig theory: the irreducible character degrees and their
  multiplicities can be expressed by polynomials in one and the same
  parameter~$q$.  Concrete estimates are given
  in~\cite[Table~1]{LS05},\footnote{The information provided
    in~\cite[Table~1]{LS05} is not entirely clear; for instance, it
    appears that some unipotent characters of small degree are not
    accommodated for; compare with \cite[Tables 4.15 and~4.16]{GM20}.
    But this does not change the bigger picture.}  based on
  \cite{DeLi85}.

  (b) Alternating groups or finite simple groups of Lie type arising
  from~\cref{thm:tits}.  Due to \cref{cor:upper-bound-unbounded-rank},
  it suffices to deal with finite simple groups of Lie type and of
  uniformly bounded Lie rank.  This allows for only finitely many Lie
  types and we may even prescribe the Lie type~$\lambda$.  In this
  situation, $\widetilde{S} = L_\lambda(q)$ for some prime power~$q$
  and \cite[Corollary~1.4~(i)]{LS05} yields, for every
  $\sigma \in \R_{>1}$,
  \[
    \zeta_{L_\lambda(q)}(\sigma) \to 1 \qquad \text{for
      $q \to \infty$.}
  \]
Taking $\sigma = 3/2$, this yields, for all but finitely
    many~$q$, the bound $r_n(\widetilde{S}) \le n^2$ for all
    $n \in \N$.  As explained above, we can accomodate for the
    finitely many exceptions, arising for small values of $q$ and the
    given Lie type~$\lambda$.
\end{proof}

Extending the proof of \cite[Theorem~6.4]{Gar16}, we obtain the next
result, which is reminiscent of a similar phenomenon in subgroup
growth, namely: every profinite group with polynomial subgroup growth
is finitely generated; compare with \cite[Theorem~10.6]{LS03}.

\begin{proposition} \label{thm:PRG-then-fg} Let $G$ be a
  quasi-semisimple profinite group with \PRG{}.  Then~$G$ is finitely
  generated.
\end{proposition}

\begin{proof}
  By \cref{pro:center-fratt}, the central elements of $G$ are
  `non-generators', i.e., contained in the Frattini subgroup.  Hence
  we may assume without loss of generality that $\mathrm{Z}(G) = 1$,
  i.e., that $G$ is semisimple:
  \[
    G \cong \prod\nolimits_{S \in \mathcal{S}} S^{\nu(S)},
   \]
   where $\mathcal{S}$ is a set of representatives of the isomorphism
   classes of non-abelian finite simple groups and $\nu(S)$ denotes
   the multiplicity of $S$ as a composition factor of~$G$.  By
   \cref{thm:PRG-iff-m(n)-polynomial}, there exists $b \in \N$ such
   that $m_n(G)\le n^b$ for all $n \in \N$.  Each $S \in \mathcal{S}$
   has a non-trivial irreducible representation of dimension less than
   $\lvert S \rvert$, and we deduce that
   $\nu(S) \le m_{\lvert S \rvert}(G) \le \lvert S \rvert^b$.  By
   \cref{thm:Wiegold-bound}, the group $S^{\nu(S)}$ is generated by
   $b+2$ elements.  This implies that the profinite group $G$ is
   generated by $b+2$ elements; compare with~\cite[Proposition~4.1.1
   and Lemma~4.1.5]{Wil98}.
\end{proof}

\begin{example}
  By \cref{thm:Wiegold-bound} and
  \cref{thm:PRG-iff-m(n)-polynomial}, the semisimple profinite group
  $G = \prod_{n \ge 5} \Alt(n)^{n!/2}$ is generated by $3$ elements,
  but does not have \PRG{}.
\end{example}

Gathering our results up to this point, we obtain a proof of
Theorem~\ref{thm:criteria-rigid-PRG}.  We conclude the section with a
refinement of our analysis in the proof of
\cref{thm:PRG-iff-m(n)-polynomial}.  This yields the main part
of~\cref{thm:growing-ranks-2-gens}.

\begin{proof}[Proof of the first part of \cref{thm:growing-ranks-2-gens}]
  Due to \cref{pro:center-fratt}, there is no harm in working modulo
  $\mathrm{Z}(G)$, and consequently we suppose that $G$ is a
  semisimple profinite group.  We choose a set of representatives
  $\mathcal{S}(G)$ for the isomorphism classes of non-abelian finite
  simple groups that occur as composition factors of~$G$.  This gives
  \begin{equation} \label{equ:G-as-cart-product} G \cong
    \prod\nolimits_{S \in \mathcal{S}(G)} S^{\nu(S)},
  \end{equation}
  where $\nu(S) \in \N$ denotes the multiplicity as a composition
  factor for each $S \in \mathcal{S}(G)$.  Moreover, refining the
  product in \eqref{equ:G-as-cart-product}, we can write $G$ as a
  cartesian product of simple factors, and the open normal subgroups
  $N \trianglelefteq_\mathrm{o} G$ are precisely the subproducts over
  all but finitely many simple factors.

  As $G$ has growing Lie ranks, we have
  $\rk_{\mathrm{Lie}}(S) \to \infty$ for $S \in \mathcal{S}(G)$ as
  $\lvert S \rvert \to \infty$ and, passing to an open subgroup if
  necessary, we may suppose that $\rk_{\mathrm{Lie}}(S) > 8$ for every
  $S \in \mathcal{S}(G)$ so that each $S$ is an alternating group or a
  classical group.  Furthermore, the minimal number of generators of
  $G$ is given by
  \[
    d(G) = \sup \bigl\{ d \bigl( S^{\nu(S)} \bigr) \mid S \in
    \mathcal{S}(G) \bigr\},
  \]
  and our description of open normal subgroups of $G$ shows that
  $d(N) \le d(G)$ for all $N \trianglelefteq_\mathrm{o} G$; compare
  with~\cite[Lemma~4.1.5]{Wil98}.  Thus it suffices to prove that
  \[
    d \bigl( S^{\nu(S)} \bigr) \le 2 \qquad \text{for all but finitely
      many $S \in \mathcal{S}(G)$.}
  \]

  At this stage we require some facts and theorems about generating
  pairs of finite alternating and classical finite simple groups~$S$,
  some of which rely on the CFSG.  It is easily seen that $\Aut(S)$
  acts faithfully on the set
  $\mathrm{GP}(S) = \{ (x,y) \in S^2 \mid S = \langle x,y \rangle \}$.
  As finite simple groups are $2$-generated, this yields
  \[
    d(S^{\nu(S)}) \le d(S) = 2 \qquad \text{as long as}
    \quad \nu(S) \le \lvert \mathrm{GP}(S) \rvert / \lvert \Aut(S)
    \rvert;
  \]
  compare with \cite[Corollary~7]{KaLu90}.  By \cite{Di69,KaLu90}, the
  probability of generating the alternating or classical finite simple
  group $S$ by two random elements tends to $1$ as $\lvert S \rvert$
  tends to infinity.  This means that
  $\lvert \mathrm{GP}(S) \rvert / \lvert S \rvert^2 \to 1$ as
  $\lvert S \rvert \to \infty$.  Furthermore, it is known that
  $\lvert \Aut(S) \rvert / \lvert S \rvert \to 1$ as
  $\lvert S \rvert \to \infty$; compare with \cite{St60}.  Hence it
  suffices to prove that
  \begin{equation} \label{equ:claim-for-nu}
    \nu(S) \le \tfrac{1}{2} \lvert S \rvert \qquad \text{for all but
      finitely many $S \in \mathcal{S}(G)$.}
  \end{equation}
  Clearly, it is enough to establish the claim separately for
  alternating groups and for classical finite simple groups.  We argue
  by contradiction.  For this we assume that at least one of the
  groups
  \[
    G_{\mathrm{alt}} = \prod_{\substack{S \in \mathcal{S}(G) \text{
          alternating} \\ \text{with } \nu(S) > \frac{1}{2} \lvert S
        \rvert}} S^{\nu(S)} \qquad \text{and} \qquad G_{\mathrm{cla}}
    = \prod_{\substack{S \in \mathcal{S}(G) \text{ classical} \\
        \text{with } \nu(S) > \frac{1}{2} \lvert S \rvert}} S^{\nu(S)}
  \]
  is infinite.  Since $G_{\mathrm{alt}}$ and $G_{\mathrm{cla}}$ are
  homomorphic images of~$G$, we reach the required contradiction, once
  we show that one of them  does not have \PRG{}.
  
  \smallskip
  
  \noindent \textit{Case 1:} $G_{\mathrm{alt}}$ is infinite.  For
  every $n \in \N$ with $n \ge 5$, the alternating group $S = \Alt(n)$
  has a non-principal irreducible character $\chi$ of degree
  $\chi(1)=n-1$.  Hence a $k$-fold direct power $S^k$ has at least $k$
  irreducible characters of degree at most~$n$, and for
  $\nu(S) > \frac{1}{2} \lvert S \rvert = \frac{1}{4} n!$ this gives
  $R_n(G) \ge R_n(S^{\nu(S)}) \ge \frac{1}{4} n!$ which grows faster
  than polynomially in~$n$.  Thus $G_{\mathrm{alt}}$ does not have
  \PRG{}.

  \smallskip

  \noindent \textit{Case 2:} $G_{\mathrm{cla}}$ is infinite.  We claim
  that $\alpha(G_{\mathrm{cla}}) \ge m/4$ for every $m \in \N$ so that
  $G_{\mathrm{cla}}$ cannot have \PRG{}.

  Let us consider a classical finite simple group $S = S_\lambda(q)$
  with associated root system~$\Phi$ of rank~$\ell$, say.  For our
  purposes the following upper bound for the minimal degree of a
  non-principal irreducible character of~$S$ is good enough: $S$ has a
  non-principal irreducible character $\chi$ satisfying
  \begin{equation} \label{equ:l-chi-order-claim}
    \ell \, \log_q(\chi(1)) \le 4 \, \log_q \lvert S \rvert.
  \end{equation}
  Indeed, if $S$ is of Chevalley type $\mathsf{A}_\ell$, respectively
  ${}^2 \mathsf{A}_\ell$, then
  $S \cong \mathsf{PSL}_{\ell+1}(\mathbb{F}_q)$, respectively
  $S \cong \mathsf{PSU}_{\ell+1}(\mathbb{F}_q)$.  Accordingly $S$ acts
  faithfully on the projective space
  $\mathbbm{P}^{\ell }(\mathbb{F}_q)$, respectively
  $\mathbbm{P}^{\ell }(\mathbb{F}_{q^2})$.  This permutation
  representation of degree
  $\lvert \mathbbm{P}^{\ell }(\mathbb{F}_q) \rvert$, respectively
  $\lvert \mathbbm{P}^{\ell }(\mathbb{F}_{q^2}) \rvert$, gives rise to
  a complex linear representation of dimension at most $q^{2\ell+2}$
  which in turn has an non-principal irreducible constituent.  By
  considering the order of a Sylow subgroup, we obtain
  $\lvert S \rvert \ge q^{\lvert \Phi^+ \rvert} = q^{\ell
    (\ell+1)/2}$. Hence \eqref{equ:l-chi-order-claim} holds in these
  cases. When $S$ is of Chevalley type $\mathsf{B}_\ell$,
  $\mathsf{C}_\ell$ or $\mathsf{D}_\ell$, the group $S$ is of
  symplectic or orthogonal type and embeds into
  $\mathsf{PSL}_k(\mathbb{F}_q)$ for $k \in \{2\ell, 2\ell +1\}$; for
  a finer analysis see~\cite{Co78}.  Hence $S$ has a non-principal
  irreducible character $\chi$ such that $\chi(1) \le q^{2\ell + 1}$.
  With the lower bound
  $\lvert S \rvert \ge q^{\lvert \Phi^+ \rvert} \ge q^{(\ell -1)\ell}$
  we see that \eqref{equ:l-chi-order-claim} holds also in these cases.

  Let $\rho$ be a non-trivial irreducible representation of $S$ with
  character $\chi$ satisfying~\eqref{equ:l-chi-order-claim}, and let
  $k, m \in \N$ be such that $\min \{k, \ell\} \ge m$.  The
  representation $\rho^{\boxtimes m}$ of the $m$-fold direkt power
  $S^m$ has dimension
  $\chi(1)^m \le \chi(1)^\ell \le \lvert S \rvert^4$.  Furthermore,
  the irreducible representations of the group $S^k$ that factor
  through one of the $\binom{k}{m}$ quotients isomorphic to $S^m$ and
  arise as pullbacks of $\rho^{\boxtimes m}$ are distinct.  This gives
  rise to the estimate
  \[
    R_{n(S)}(S^k) \ge \binom{k}{m} \ge \frac{(k-m)^m }{m !}, \quad
    \text{where we write $n(S) = \lvert S \rvert^4$ for `crispness'.}
  \]

  Now we fix $m \in \N$ and obtain, subject to the conditions
  $\nu(S) > \frac{1}{2} \lvert S \rvert$ and
  $\ell = \rk_{\mathrm{Lie}}(S) \ge m$,
  \[
    \frac{\log (R_{n(S)} (G))}{\log (n(S))} \ge \frac{\log \bigl(R_{n(S)} (
      S^{\nu(S)}) \bigr)}{4 \log (\lvert S \rvert)} \ge \frac{m \log \bigl(
      \tfrac{1}{2} \lvert S \rvert - m \bigr) - \log(m !)}{4 \log \lvert S
      \rvert} \to \frac{m}{4} \quad \text{as
      $\lvert S \rvert \to \infty$.}
  \]
  From this we deduce that $\alpha(G_{\mathrm{cla}}) > m/4$ as
  required.
\end{proof}

% The second part of \cref{thm:growing-ranks-2-gens} is justified
% by \cref{exa:prod-of-SL2-with-presc-gens}.
  
%%%

\section{Approximating representation zeta functions of \\ cartesian
  products of groups} \label{sec:Aproox}

In this section we establish the remaining part of
\cref{thm:growing-ranks-2-gens} and
\cref{thm:unif-bounded-ranks-invariant-alpha}.  From
\cite[Section~2]{AKOV16} we recall a specific technique for
approximating representation zeta functions by means of suitable
infinite products of Dirichlet polynomials.  In this context the term
`Dirichlet polynomial' refers to a Dirichlet series
$\sum_{n=1}^\infty c_n n^{-s}$ with only finitely many non-zero
coefficients $c_n \in \C$; the representation zeta functions of finite
groups yield typical examples of interest to us.  For Dirichlet
polynomials $f, g$ with non-negative integer coefficients and
for~$C \in \R$ we write $f  \sim_C g$ if
\[
  f(\sigma) \le C^{1+\sigma} g(\sigma) \quad \text{and} \quad
  g(\sigma) \le C^{1+\sigma} f(\sigma) \qquad \text{for every
    $\sigma \in \R_{>0}$.}
\]
We denote by $\mathcal{A}^+$ the collection of all finite subsets
$\boldsymbol{a} \subset \N_0 \times \N$.  For
$\boldsymbol{a} \in \mathcal{A}^+$ and $q\in \N_{\ge 2}$ we define the
`basic' Dirichlet polynomial
\begin{equation}
  \label{equ:Dirichlet-polynominal}
  \xi_{\boldsymbol{a},q}(s)=\sum_{(m,n)\in \boldsymbol{a}}q^{m-ns}.
\end{equation}
For any choice of $\boldsymbol{a}_i \in \mathcal{A}^+$ and $q_i \in \N_{\ge 2}$,
for $i \in \N$, such that $q_i \to \infty$ as $i \to \infty$, the
product
\[
  \prod\nolimits_{i \in \N} \bigl( 1 + \xi_{\boldsymbol{a}_i,q_i} \bigr)
\]
can be regarded as a Dirichlet series and we denote its abscissa of
convergence by
\[
  \absc \Bigl( \prod\nolimits_{i \in \N} ( 1 + \xi_{\boldsymbol{a}_i,q_i} ) \Bigr).
\]

The following two auxiliary propositions follow easily, for instance,
from the discussion in \cite[Section~2]{AKOV16}.

\begin{proposition} \label{pro:abs-of-direct-product} The
  representation zeta function of a direct product $G \times H$ of two
  representation rigid profinite groups has abscissa of convergence
  $\alpha(G \times H) = \max \{ \alpha(G), \alpha(H) \}$.
\end{proposition}

\begin{proposition} \label{pro:auxiliary-result-for-alpha} Let $G$ be
  a profinite group that decomposes as a cartesian product
  $G = \prod_{i \in \N} G_i$ of finite groups, and suppose that
  $H \le G$ is a subgroup of the form $H = \prod_{i \in \N} H_i$,
  where $H_i \le G_i$ with $H_i = [H_i,H_i]$ for each $i \in \N$.

  Suppose that there are $C \in \R$ as well as $\boldsymbol{a}_i \in \mathcal{A}^+$
  and $q_i \in \N_{\ge 2}$, for $i \in \N$, such that
  $\zeta_{G_i} - \lvert G_i : [G_i,G_i] \rvert \sim_C \xi_{\boldsymbol{a}_i,q_i}$.
  Suppose further that $q_i \to \infty$ as $i \to \infty$ and
  that $\lvert G_i : H_i \rvert \le C$ for all $i \in \N$.

  Then the representation zeta function of $H$ has abscissa of
  convergence
  \[
    \alpha(H) = \absc\Bigl( \prod\nolimits_{i \in \N} (1 + \xi_{\boldsymbol{a}_i,q_i} )
    \Bigr).
  \]
\end{proposition}

We illustrate these ideas and justify the remaining part of
\cref{thm:growing-ranks-2-gens}.

\begin{example} \label{exa:prod-of-SL2-with-presc-gens} The complex representations of the special linear group
  $\SL_2(\mathbb{F}_q)$ over a finite field $\mathbb{F}_q$ are well
  understood; for instance, see \cite[Section~12.5]{DM20}. In
  particular, we have
  $\zeta_{\SL_2 (\mathbb{F}_q)}(s) - 1 \sim_2 q^{1-s}$ for suﬃciently
  large~$q$.

  Let $d \in \N$ with $d \ge 3$.  Then \cref{thm:Wiegold-bound}
  implies that the quasi-semisimple profinite group
  \[
    G = \prod_{p \in \mathbb{P}_{\ge 3}}
    \SL_2(\mathbb{F}_p)^{m(p)}, \quad \text{where} \;\; m(p)
    = \lvert \SL_2(\mathbb{F}_q) \rvert^{d-2} = \bigl(
    (p^3-p)/2 \bigr)^{d-2},
  \]
  requires $d$ generators, and so does every open normal subgroup
  $N \trianglelefteq_\mathrm{o} G$.  Furthermore,
  \cref{pro:auxiliary-result-for-alpha}, for $G = H$, yields that $G$
  has \PRG{} with
  \begin{multline*}
    \alpha(G) = \absc\Bigl( \prod\nolimits_{p \in \mathbb{P}_{\ge 3}}
    (1 + p^{1-s})^{ ( (p^3-p)/2 )^{d-2}} \Bigr) = \absc\Bigl(
    \prod\nolimits_{p \in \mathbb{P}_{\ge 3}} (1 + p^{3d-5-s}) \Bigr)
    \\
    = 3d-4.
  \end{multline*}
\end{example}

Next we recall \cite[Theorem~3.1]{AKOV16}, in a specialised form that
fits our setting and aims.  The notion of a Lie type was described in
\cref{def:lie-type-basic}.  Furthermore, for an irreducible root
system~$\Phi$, let $\boldsymbol{\Lambda}_\Phi$ denote the collection
of Lie types $\lambda$ such that~$\Phi_\lambda \cong \Phi$.  Finally,
let $\mathcal{Q}$ denote the set of all prime powers.

\begin{theorem}[Avni, Klopsch, Onn, Voll] \label{thm:approximation}
  Let $\Phi$ be an irreducible root system.  Then there exist
  $C \in \R$, a finite set $\boldsymbol{a}(\Phi)\in \mathcal{A}^+$,
  and subsets
  $\boldsymbol{a}(\lambda,q) \subseteq \boldsymbol{a}(\Phi)$ for
  $(\lambda,q)\in \boldsymbol{\Lambda}_\Phi\times \mathcal{Q}$ whose
  union is $\boldsymbol{a}(\Phi)$ such that the following holds.

  For every $q \in \mathcal{Q}_{>3}$ and every finite group
  $\mathbf{G}^F\!\!$ of Lie type
  $\lambda \in \boldsymbol{\Lambda}_\Phi$, where $\mathbf{G}$ is a
  connected simple $\overline{\mathbb{F}_q}$-algebraic group and $F$
  is a Frobenius endomorphism defining an $\mathbb{F}_q$-structure, we
  have
  \[
    \zeta_{\mathbf{G}^F} - \big\vert \mathbf{G}^F :
      [\mathbf{G}^F, \mathbf{G}^F] \big\vert \, \sim_C \,
    \xi_{\boldsymbol{a}(\lambda,q),q}.
  \]
  Moreover,
  $(\rk(\Phi), \lvert \Phi^+ \rvert) \in \boldsymbol{a}(\Phi)$ and
  \[
    \zeta_{\mathbf{G}^F} \, \sim_C \, 1 + q^{\rk(\Phi) - \lvert \Phi^+
      \rvert s}.
  \]
\end{theorem}

We remark that the final assertion in \cref{thm:approximation} implies that
\[
  % m \le \rk(\Phi), \quad n \le \lvert \Phi^+ \rvert, \quad
  m/n \le \rk(\Phi) / \lvert \Phi^+ \rvert \quad \text{for every
    $(m,n) \in \boldsymbol{a}(\Phi)$.}
\]
With these preparations we are ready to deduce
\cref{thm:unif-bounded-ranks-invariant-alpha}.

\begin{proof}[Proof of \cref{thm:unif-bounded-ranks-invariant-alpha}]
  As $\widetilde{G}$ surjects onto~$G$, we have
  $\alpha(G) \le \alpha(\widetilde{G})$ and it suffices to prove the
  reverse inequality.  We may suppose that $G$ is infinite and
  semisimple with $\alpha(G) < \infty$ so that we can write
  $G \cong \prod_{i \in \N} S_i$ and
  $\widetilde{G} \cong \prod_{i \in \N} \widetilde{S}_i$, where each
  $S_i$ is a non-abelian finite simple group with universal cover
  $\widetilde{S}_i$.

  Using \cref{pro:abs-of-direct-product}, we see that the abscissae
  $\alpha(G)$ and $\alpha(\widetilde{G})$ do not change if we remove
  finitely many factors $S_i$ respectively $\widetilde{S}_i$.
  Moreover, we can equally ignore all (possibly infinitely many!)
  factors $S_i$ respectively $\widetilde{S}_i$ that are Suzuki or Ree
  groups, since these satisfy $S_i \cong \widetilde{S}_i$; see
  \cite[Corollary~24.13]{MT11}.  As $G$ has bounded Lie rank, we may
  thus assume that all the groups $S_i$ are of Lie type and, more
  specifically, that $S_i$ and $\widetilde{S}_i$ are as described in
  \cref{thm:tits} and the remark following it:
  \[
    S_i = \mathbf{G}_i^{\, F_i} / \mathrm{Z}(\mathbf{G}_i^{\, F_i})
    \qquad \text{and} \qquad \widetilde{S}_i = \mathbf{G}_i^{\, F_i}
  \]  
  for simply connected simple groups $\mathbf{G}_i$ of bounded rank
  equipped with Frobenius endomorphisms~$F_i$.

  Clearly, there are only finitely many irreducible root systems of
  bounded rank.  Invoking \cref{pro:abs-of-direct-product} once more,
  it suffices to consider the case that all the groups~$\mathbf{G}_i$
  have the same root system~$\Phi$ and that $q_i = q_{F_i} \ge 3$.  We
  pick $C \in \R$ and
  $\boldsymbol{a} = \boldsymbol{a}(\Phi) \in \mathcal{A}^+$ with the
  properties described in~\cref{thm:approximation}; clearly, we can
  arrange that $C \ge \rk(\Phi)+1$.  We write
  $(\mathbf{G}_i)_\mathrm{ad}$ to denote the adjoint form of
  $\mathbf{G}_i$ and write
  $\pi \colon \mathbf{G}_i \to (\mathbf{G}_i) _\mathrm{ad}$ for the
  central isogeny.  From \cite[Proposition~24.21]{MT11} we recall that
  \[
    S_i \cong \pi(\mathbf{G}_i^{\, F_i}) =
    [(\mathbf{G}_i)_{\mathrm{ad}}^{\, F_i},
    (\mathbf{G}_i)_{\mathrm{ad}}^{\,
      F_i}],
  \]
  since we are in the setting of \cref{thm:tits} with $q>3$.
  Moreover, we have
  \[
    \big\vert (\mathbf{G}_i)_\mathrm{ad}^{\, F_i} :
    [(\mathbf{G}_i)_\mathrm{ad}^{\, F_i},
    (\mathbf{G}_i)_\mathrm{ad}^{\, F_i}] \big\vert = \lvert
    \mathrm{Z}(\mathbf{G}_i^{\, F_i}) \rvert \le \rk(\Phi)+1 \le C
    \qquad \text{for all $i \in \N$;}
  \]
  see \cite[Corollary~24.13]{MT11}.

  \cref{thm:approximation} yields for each $i \in \N$ a suitable
  $\boldsymbol{a}_i \subseteq \boldsymbol{a}$ and
  $q_i \in \mathcal{Q}$ with $q_i \to \infty$ as $i \to \infty$ such
  that
  \[
    \zeta_{(\mathbf{G}_i)_\mathrm{ad}^{F_i}} - \big\vert
    (\mathbf{G}_i)_\mathrm{ad}^{\, F_i} :
    [(\mathbf{G}_i)_\mathrm{ad}^{\, F_i},
    (\mathbf{G}_i)_\mathrm{ad}^{\, F_i}] \big\vert \,\sim_C\,
    \zeta_{\mathbf{G}_i^{F_i}} - 1 \,\sim_C\, \xi_{\boldsymbol{a}_i,q_i}.
  \]
  Hence \cref{pro:auxiliary-result-for-alpha} implies that
  \[
    \alpha(G) = \absc \Bigl( \prod\nolimits_{i \in \N} (1
    +\xi_{\boldsymbol{a}_i,q_i}) \Bigr) =
    \alpha(\widetilde{G}). \qedhere
  \]
\end{proof}

%%%%%%%%%%%%%%%%%%%%%%%%%%%%%%%%%%%%%%%%%%%%%%

\section{Groups with specified representation growth
  rates} \label{sec:proofs}

In this section we prove \cref{thm:main-result}, with its
\cref{cor:prof-compl,cor:growing-rank-invariant-alpha}.  First we
establish \cref{thm:for-every-a-we-find-G-that-has-PRG}, by adapting
the proof of \cite[Theorem~6.3]{Gar16} to our more general setting.
To complete the argument we state and prove
\cref{thm:main-constr-for-given-rho}, in a form that is somewhat more
general than needed here.  To start with, we extend
Definition~\ref{def:lie-type-basic}, to coin the notion of a
quasi-semisimple profinite group of a prescribed Lie type.

\begin{definition} \label{def:lie-type-profinite} Let
  $\lambda = (\Phi,\tau)$ be a Lie type and fix a prime~$p$.  Let
  $\mathcal{Q}_p \subseteq \mathcal{Q}$ denote the set of all powers
  of~$p$, excluding $2$ and~$3$.  Fix a connected simply connected
  simple algebraic group $\mathbf{G}$ defined over an algebraically
  closed field of characteristic~$p$ with associated root
  system~$\Phi$.  We consider a family
  $\mathcal{F} = (F_i)_{i \in \N}$ of Frobenius endomorphisms $F_i$
  defining $\mathbb{F}_{q_i}$-structures of $\mathbf{G}$ with
  $q_i \in \mathcal{Q}_p$ and inducing the automorphism~$\tau$ on the
  root system~$\Phi$.  This yields a family
  $L_\lambda(q_i) = \mathbf{G}^{F_i}$, $i \in I$, of quasi-simple
  groups of Lie type~$\lambda$ with simple parts
  $S_\lambda(q_i ) = L_\lambda(q_i) / \mathrm{Z}(L_\lambda(q_i))$.  As
  recorded in \cref{rem:covering-fin-exceptions}, each
  $L_\lambda(q_i)$ is the universal covering group of
  $S_\lambda(q_i )$, apart from finitely many exceptions.

  If $G$ is a quasi-semisimple profinite group with semisimple part
  \[
    G/\mathrm{Z}(G) \cong \prod\nolimits_{i \in \N} S_\lambda(q_i),
  \]
  we say that $G$ is a quasi-semisimple profinite group \emph{in
    characteristic $p$} and \emph{of Lie type~$\lambda$}.  To be even
  more specific we say that $G$ is \emph{modelled on $\mathcal{F}$}.
  Furthermore, we call $\rk(G) = \rk(\mathbf{G}) = \rk(\Phi)$ the
  \emph{Lie rank} of $G$.
\end{definition}

For the following preparatory lemma we recall that $\mathcal{A}^+$
denotes the collection of all finite subsets of $\N_{0} \times\N$ and
that in \eqref{equ:Dirichlet-polynominal} we defined Dirichlet
polynomials $\xi_{\boldsymbol{a},q}(s)$ for
$\boldsymbol{a} \in \mathcal{A}^+$ and $q\in \N_{\ge 2}$.
  
\begin{lemma} \label{lem:f(j)-choice-for-rho} Let
  $\varrho_0, \varrho \in \R$ be such that $0 < \varrho_0 < \varrho$,
  and let $q \in \N$ with $q \ge 2$.  Let
  $\boldsymbol{a} \in \mathcal{A}^+$ and, for
  $j \in \N$, let
  $\varnothing \ne \boldsymbol{a}_j \subseteq \boldsymbol{a}$ be such
  that
  \[
    m/n \le \varrho_0 \quad \text{ for all $(m,n) \in \boldsymbol{a}$.}
  \]
  Then the relation $\prec$ defined by
  \[
    (m,n) \prec (m',n') \qquad \text{if} \quad (m-n \varrho > m' - n'
    \varrho) \,\lor\, (m-n \varrho = m' - n' \varrho \,\land\, n <
    n'),
  \]
  for $ (m,n), (m',n') \in \boldsymbol{a}$, gives a strict linear
  order on~$\boldsymbol{a}$.  Let $ (m_0,n_0) \in \boldsymbol{a}$ be
  $\prec$-minimal subject to
  \[
    \{ j \in \N \mid (m_0,n_0) \in \boldsymbol{a}_j \} \subseteq \N
  \]
  being infinite, let $(k_j)_{j\in \N}$ be a sequence of natural
  numbers such that
  $\big\lvert \tfrac{1}{j} k_j - \varrho \big\rvert \le 1/j$ for all
  $j\in \N$ and set
  \[
    f(j) =
    \begin{cases}
      n_0 k_j  - m_0 j & \text{if  $j \ge (\varrho - \varrho_0)^{-1}$,}\\
      0 & \text{otherwise}
    \end{cases}
    \qquad \text{for $j \in \N$,}
  \]
  
  Then the Dirichlet series corresponding to the infinite
  product
  \begin{equation}\label{equ:prod-of-xis}
    \prod_{j \in \N} \big( 1 + \xi_{\boldsymbol{a}_j,q^j} \big)^{q^{f(j)}}
  \end{equation}
  has abscissa of convergence $\varrho$.
\end{lemma}

\begin{proof}
  It is easy to check that $\prec$ yields a strict linear order on
  $\boldsymbol{a}$ as claimed and that $f(j) \ge 0$ for all
  $j \in \N$.  The Dirichlet series corresponding to the infinite
  product in~\eqref{equ:prod-of-xis} has non-negative integral
  coefficients.  Thus it is enough to look at the convergence
  properties of the product for $s = \sigma \in \R_{>0}$ close to
  $\varrho$.  The product converges if and only if the sum
  \begin{equation} \label{equ:sum-ofq-powers} \sum_{j \in \N} q^{f(j)}
    \sum_{(m,n) \in \boldsymbol{a}_j} q^{j(m-n\sigma)}
  \end{equation}
  converges.  Due to our choice of $(m_0,n_0)$, we may -- by ignoring
  finitely many summands -- suppose without loss of generality that
  for all sufficiently small $\varepsilon\ge 0$,
  \[
    m_0 - n_0 (\varrho + \varepsilon) \ge m-n (\varrho +
    \varepsilon) \quad \text{for all $j\in \N$ and $(m,n) \in \boldsymbol{a}_j$.}
  \]
  
  First suppose that $\sigma = \varrho + \varepsilon$ for such a small
  $\varepsilon > 0$.  We put $\delta = - n_0 \varepsilon/2 <0$ and
  $j_0 = \lceil \max \{ (\varrho - \varrho_0)^{-1}, 2/\varepsilon \}
  \rceil$.  For $j \ge j_0$ we see that
  \begin{multline*}
    \tfrac{1}{j} f(j) + (m_0 - n_0 \sigma) = \bigl( n_0 \tfrac{1}{j} k_j -
    m_0 \bigr) + (m_0 - n_0 \sigma) \le n_0 ( \varrho + \tfrac{1}{j}) - n_0
    (\varrho + \varepsilon) \\ = n_0 
    \bigl( \tfrac{1}{j} - \varepsilon \bigr) \le \delta.
  \end{multline*}
  From this we deduce that 
  \[
    \sum_{j \ge j_0} q^{f(j)} \sum_{(m,n) \in \boldsymbol{a}_j}
    q^{j(m-n\sigma)} \le \lvert \boldsymbol{a} \rvert \sum_{j \ge j_0}
    \left( q^{\frac{1}{j} f(j) + (m_0 - n_0 \sigma)} \right)^j <
    \lvert \boldsymbol{a} \rvert \sum_{j \ge j_0} (q^\delta)^j <
    \infty
  \]
  and thus the sum in~\eqref{equ:sum-ofq-powers} is convergent.
  
  Next suppose that $\sigma = \varrho - \varepsilon$ for small
  $\varepsilon > 0$.  Due to our choice of $(m_0,n_0)$, the set
  $J = \{ j \in \N \mid (m_0,n_0) \in \boldsymbol{a}_j \} $ is infinite, and it is
  enough to show that
  \[
    \sum_{j \in J} q^{f(j)} q^{j (m_0 - n_0 \sigma)} = \sum_{j \in J}
    \left( q^{\frac{1}{j} f(j) + (m_0 - n_0 \sigma)} \right)^j
  \]
  diverges.  For this it suffices to observe that for all
  $j \ge \lceil \max \{ (\varrho - \varrho_0)^{-1}, \varepsilon^{-1}
  \} \rceil$,
  \begin{multline*}
    \tfrac{1}{j} f(j) + (m_0 - n_0 \sigma) = n_0 \tfrac{1}{j} k_j -
    m_0 + (m_0 - n_0 \sigma) \ge n_0 ( \varrho -
    \tfrac{1}{j}) - n_0 (\varrho - \varepsilon) \\
    = n_0 \bigl( \varepsilon -\tfrac{1}{j} \bigr) \ge 0. \qedhere
  \end{multline*}
\end{proof}

\begin{theorem} \label{thm:for-every-a-we-find-G-that-has-PRG} Let
  $\varrho \in \R_{>0}$.  Let $\lambda = (\Phi,\tau)$ be a Lie type
  such that $\varrho > \rk(\Phi) / \lvert \Phi^+ \rvert$.  Let $q$ be
  a power of a prime~$p$.  Then there exists a function
  $f \colon \N \to \N_0$ such that every quasi-semisimple profinite
  group $G$ in characteristic $p$ and of Lie type $\lambda$ with
  semisimple part
  $G / \mathrm{Z}(G) \cong \prod_{j \in \N} S_\lambda(q^j)^{q^{f(j)}}$
  has \PRG{} of degree $\alpha({G}) = \varrho$.
\end{theorem}

\begin{proof}
  Without loss of generality, we may stay clear of the finitely many
  exceptional cases referred to in \cref{thm:tits} and
  \cref{rem:covering-fin-exceptions}.  In particular, this means that
  in the argument below we tacitly require that $j \ge 2$, for certain
  $\lambda$ and $q \in \{2,3\}$.  By
  \cref{thm:unif-bounded-ranks-invariant-alpha} it is enough to
  consider the case that $G$ is the universal covering group, that is,
  we may work with
  \[
    G = \prod\nolimits_{j \in \N} L(q^j)^{q^{f(j)}};
  \]
  compare with \cref{def:lie-type-basic}.
  
  We fix $C \in \R$, a finite set
  $\boldsymbol{a} = \boldsymbol{a}(\Phi) \in
    \mathcal{A}^+$ and subsets
  $\boldsymbol{a}_j = \boldsymbol{a}(\lambda,q^j) \subseteq
  \boldsymbol{a}$ for $j \in \N$ such that the conclusions collected
  in \cref{thm:approximation} hold.  We choose $f(j)$ as in
  \cref{lem:f(j)-choice-for-rho}, for
    $\varrho_0 = \rk(G)/\lvert \Phi^+ \rvert$, and apply
  \cref{pro:auxiliary-result-for-alpha}, for $G=H$, to conclude the
  proof.
\end{proof}

By design, the groups supplied by
\cref{thm:for-every-a-we-find-G-that-has-PRG} have bounded, in fact
constant Lie rank.  In order to produce quasi-semisimple groups of
prescribed representation growth that are also profinite completions,
we need to take one further step; compare with
\cref{thm:profinite-completion} below.

\begin{theorem} \label{thm:main-constr-for-given-rho}
  Let $(\varrho_m)_{m\in \N}$ be a strictly increasing sequence of
  positive real numbers converging to a number~$\varrho \in \R_{>0}$.
  Suppose that for each $m \in \N$ there is a
  semisimple profinite group of Lie type
  \[
    H_m = \prod\nolimits_{j=1}^\infty S_{m,j},
  \]
  whose composition factors $S_{m,j}$ are of common Lie type
  $\lambda_m$, such that $\alpha(H_m) = \varrho_m$.  Furthermore,
  suppose that $\{ \rk(H_m) \mid m \in \N \}$ is unbounded.

  Then there is a finitely generated semisimple profinite group of the
  form
  \[
    G = \prod\nolimits_{m=1}^\infty K_m,
  \]
  where each $K_m $ is a suitable finite quotient of the group~$H_m$,
  such that $G$ and its universal covering group $\widetilde{G}$
  satisfy $\alpha(G) = \alpha(\widetilde{G}) = \varrho$.
\end{theorem}

\begin{proof}
  Without loss of generality, we may assume that the Lie rank
  $\rk(H_m)$ is strictly increasing with~$m$.  We set $H_0 = 1$ and
  aim to construct the desired group
  \[
    G = \prod\nolimits_{m=1}^\infty K_m \cong \varprojlim_{m \in \N}
    L_m, \quad \text{where we write } L_m = \prod\nolimits_{l=1}^m
    K_l,
  \]
  as a cartesian product of suitable finite quotients $K_m $ of the
  groups~$H_m$ such that $\alpha(G) = \alpha(\widetilde{G}) =\varrho$.
  By design, such a group $G$ has \PRG{} and is thus finitely generated
  by \cref{thm:PRG-then-fg}.

  By \cref{thm:unif-bounded-ranks-invariant-alpha}, the group $H_m$
  and its universal covering group
  $\widetilde{H}_m = \prod_{j=1}^\infty \widetilde{S}_{m,j}$ satisfy
  $\alpha(H_m) = \alpha(\widetilde{H}_m) = \varrho_m$ for every
  $m \in \N$.  In order to build a group~$G$ with the required
  representation growth, it is enough to produce, for $m \in \N_0$,
  positive integers $\underline{n}(m)$ such that
  \[
    1 = \underline{n}(0) < \underline{n}(1) < \underline{n}(2) < \dots
  \]
  and finite quotients $K_m$ of the groups $H_m$ such that the finite
  semisimple groups $L_m = \prod_{l=1}^m K_l$ and their universal
  covers $\widetilde{L}_m = \prod_{l=1}^m \widetilde{K}_l$ satisfy the following.
  \begin{enumerate}
  \item[(i)] For each $m \in \N$, all irreducible representations of
    $\widetilde{L}_m$ of dimension at most $\underline{n}(m-1)$ factor
    through the natural projection
    $\widetilde{L}_m = \widetilde{L}_{m-1} \times \widetilde{K}_m \to
    \widetilde{L}_{m-1}$; equivalently we have
    \begin{equation}\label{equ:no-new-small-reps}
      R_{\underline{n}(m-1)} \bigl(\widetilde{L}_m \bigr) = R_{\underline{n}(m-1)} \bigl(
      \widetilde{L}_{m-1} \bigr).
    \end{equation}
  \item[(ii)] For each $m \in \N$,
    \begin{equation} \label{equ:never-larger-than-rho} \frac{\log
        \bigl( R_n \bigl( \widetilde{L}_m \bigr) \bigr)}{\log n} \le
      \varrho \qquad \text{for all $n \in \N$ with $n > \underline{n}(m-1)$.}
    \end{equation}
    and
    \begin{equation}\label{equ:close-to-rho-m}
      \frac{\log \left( R_{\underline{n}(m)} (L_m ) \right)}{\log \left( \underline{n}(m)
        \right)} \ge \varrho_m - \frac{1}{m},
    \end{equation}
  \end{enumerate}
  Indeed, from these properties it follows that
  \[
    \alpha(G) = \limsup_{n \in \N} \frac{\log R_n(G)}{\log n} =
    \limsup_{m \in \N} \max_{\underline{n}(m-1)<n \le
      \underline{n}(m)}\frac{\log R_n(L_m)}{\log n}
  \]
  and similarly $\alpha(\widetilde{G})$ lie in the real interval
  $[\varrho_m - \tfrac{1}{m}, \varrho]$ for all $m \in \N$, and hence
  we obtain $\alpha(G) = \alpha(\widetilde{G}) = \varrho$.

  \smallskip
  
  We produce suitable positive integers $\underline{n}(m)$ and finite quotients
  $K_m$ recursively.  We recall that $\underline{n}(0) = 1$ and $H_0 = 1$, and we
  set $K_0=1$.  Now let $m \in \N$, and suppose that suitable positive
  integers $\underline{n}(l)$ and finite quotients $K_l$ have already been
  identified for $0 \le l \le m-1$.  Since
  $\widetilde{L}_{m-1} = \prod_{l = 1}^{m-1} \widetilde{K}_l$ is
  finite, we see that
  \begin{equation}\label{equ:rho-does-not-change}
    \alpha(\widetilde{L}_{m-1} \times \widetilde{H}_m) = \alpha(\widetilde{H}_m) =
    \varrho_m \quad \text{and} \quad     \alpha(L_{m-1} \times H_m) =
    \alpha(H_m) = \varrho_m.
  \end{equation}
  Jordan's theorem on finite linear groups implies that, by deleting
  finitely many factors $S_{m,j}$ from~$H_m$, correspondingly
  $\widetilde{S}_{m,j}$ from~$\widetilde{H}_m$, we can arrange that
  $\widetilde{H}_m$ does not have any representations of dimension at
  most $\underline{n}(m-1)$ apart from the principal $1$-dimensional
  representation.  This guarantees that \eqref{equ:no-new-small-reps}
  holds for any choice of~${K_m \le H_m}$.  Furthermore,
  \eqref{equ:rho-does-not-change} and $\varrho_m < \varrho$ imply
  that, by deleting again finitely many factors $S_{m,j}$ from~$H_m$,
  correspondingly $\widetilde{S}_{m,j}$ from~$\widetilde{H}_m$, we can
  arrange that~\eqref{equ:never-larger-than-rho} holds for any choice
  of~$K_m \le H_m$.  Finally, \eqref{equ:rho-does-not-change} implies
  that there exist $\underline{n}(m) > \underline{n}(m-1)$ and a
  finite quotient $K_m$ of $H_m$ such that \eqref{equ:close-to-rho-m}
  holds.
\end{proof}

Clearly, \cref{thm:main-result} is a direct consequence of
\cref{thm:for-every-a-we-find-G-that-has-PRG,thm:main-constr-for-given-rho}.

\begin{remark} \label{diff-to-Garcia-Klopsch} Garc\'ia Rodr\'iguez and
  Klopsch's construction in \cite[Chapter~6]{Gar16} only produced
  certain semisimple profinite groups in positive characteristic of
  fixed Lie type.  More precisely, they considered cartesian products
  of $\SL_{p^{\beta}}(p^{\gamma})$ or $\SU_{p^{\beta}}(p^{\gamma})$
  for a fixed prime power $p^\beta$ across infinitely many values of
  $\gamma$, and they indicated that, similarly, one could work with
  cartesian products of $\Sp_{2\eta}(2^{\gamma})$,
  $\Spin_{2\eta}^+(2^{\gamma})$, or $\Spin^-_{2\eta}(2^{\gamma})$ for
  fixed $\eta$ and increasing $\gamma$.  In contrast, the
  quasi-semisimple profinite groups $G$ with prescribed polynomial
  representation growth resulting from \cref{thm:main-result} can
  involve much more flexibly composition factors that are simple
  groups of Lie type of any preferred Chevalley type and underlying
  field cardinality.  In particular, this leads to groups $G$ that are
  profinite completions, as we explain next.
\end{remark}

We recall that a finitely generated profinite group $G$ is termed a
\emph{profinite completion} if there exists a finitely generated
residually finite group $\Gamma$ such that $G$ is isomorphic to the
profinite completion~$\widehat{\Gamma}$.  It is generally difficult to
decide whether a finitely generated profinite group is a profinite
completion or not; for instance, see \cite{LPS96,Seg01,Pyb03}.
In~\cite{KN06}, Kassabov and Nikolov studied this question for
semisimple profinite groups and produced the following criterion.

\begin{theorem}[Kassabov and Nikolov]
  \label{thm:profinite-completion}
  Let $G$ be a finitely generated semisimple profinite group that has
  growing Lie ranks.  Then $G$ is a profinite completion.
\end{theorem}

As far as we know it is open whether this or a similar criterion
extends to all quasi-semisimple profinite groups. However, similar
results have been recently obtained by Kionke and Piccolo
(unpublished) on cartesian products of quasi-simple groups.  For our
purposes it is of interest to observe that the complex representations
of a quasi-semisimpe profinite completion $G =\widehat{\Gamma}$ are
tightly linked to those of a finitely generated group $\Gamma$ giving
rise to~$G$.

\begin{proposition} \label{pro:factors-through-compl}
  \label{pro:repr-of-abstr-H}
  Let $\Gamma$ be a finitely generated discrete group such that its
  profinite completion~$\widehat{\Gamma}$ is quasi-semisimple.  Then
  every finite dimensional complex representation $\varrho$
  of~$\Gamma$ factors through the profinite
  completion~$\Gamma \hookrightarrow \widehat{\Gamma}$.
\end{proposition}

\begin{proof}
  The complex representations $\varrho$ of~$\Gamma$ that factor
  through the profinite
  completion~$\Gamma \hookrightarrow \widehat{\Gamma}$ are those that
  have finite image~$\Gamma \varrho$.  For a contradiction, assume
  that $\varrho \colon \Gamma \to \mathsf{GL}_n(\C)$ is a finite
  dimensional representation with infinite image
  $\Delta = \Gamma \varrho$.  Then $\Delta$ is a finitely generated
  linear group and hence residually finite.  Moreover, there exist a
  prime~$p$ and a finite-index normal subgroup
  $\Sigma \trianglelefteq \Delta$ that is residually a finite
  $p$-group; compare with \cite[Window~7, Proposition~9]{LS03}.  Thus
  $\widehat{\Delta}$ admits an open subgroup $D$ that maps onto the
  infinite pro-$p$ group $\widehat{\Sigma}_p$.  The quasi-semisimple
  profinite group $G= \widehat{\Gamma}$ maps onto $\widehat{\Delta}$
  and so does its universal cover~$\widetilde{G}$.  Write
  $G/\mathrm{Z}(G) \cong \prod_{i \in \N} S_i$, with finite
  non-abelian simple factors~$S_i$.  Then
  $\widetilde{G} \cong \prod_{i \in \N} \widetilde{S}_i $ admits an
  open subgroup, contained in the preimage of~$D$, that is perfect and
  maps onto an open subgroup of $\widehat{\Sigma}_p$.  But a
  non-trivial pro-$p$ group cannot be perfect, a contradiction.
\end{proof}

As indicated in the introduction, \cref{cor:prof-compl} is a
consequence of
\cref{thm:main-result,thm:profinite-completion,pro:factors-through-compl},
while \cref{cor:growing-rank-invariant-alpha} follows directly from
\cref{thm:main-result}.
 
%%%
\appendix
%%%

\section{Theorem~\ref{thm:approximation} revisited}
\label{sec:AKOV-revisited}

In this section we prove a variant of \cref{thm:approximation} that
applies to finite simple groups of Lie type, following closely the
strategy that was developed in \cite[Proof of Theorem~3.1]{AKOV16}.
This also provides the opportunity to supply some additional details
and to streamline parts of the argument presented
in~\cite{AKOV16}.

%%%
\subsection{Basic set-up based on Deligne--Lusztig theory}
We consider finite groups of Lie type, as introduced in
\cref{sec:finite}.  The irreducible characters of such a group~
$\mathbf{G}^F$ can be decomposed into Deligne--Lusztig series, which
can be parametrised by conjugacy classes of semisimple elements in a
dual group~$(\mathbf{G}^*)^{F^*}\!$.
We discuss briefly the most relevant ingredients;
details can be found, for instance, in
\cite{Car93,DM20,GM20}.
% \footnote{DECIDE whether we use DM or GM as a default reference, or
% both}

Let $\mathbf{G}$ be a connected reductive algebraic group, defined
over an algebraically closed field of positive characteristic, as in
\cref{sec:finite}.  Let $F$ be a Frobenius endomorphism providing an
$\mathbb{F}_q$-structure on~$\mathbf{G}$.  The dual group
$\mathbf{G}^*$, equipped with a dual Frobenius endomorphism $F^*$, is
a connected reductive group whose root datum is dual the one of~$G$.
Deligne--Lusztig theory hinges on certain virtual
characters~$R_\mathbf{T}^\mathbf{G}(\theta)$ that are associated to
irreducible characters $\theta \in \Irr(\mathbf{T}^F)$, where
$\mathbf{T}$ denotes an $F$-stable maximal torus of~$\mathbf{G}$.  Any
two such Deligne–Lusztig characters are either equal or orthogonal to
each other, see \cite[Corollary~2.2.10]{GM20}.  Moreover, if
$\mathbf{T}_1$ and $\mathbf{T}_2$ are two $F$-stable maximal tori
of~$\mathbf{G}$ and if $\theta_1 \in \Irr(\mathbf{T}_1^F)$
and~$\theta_2 \in \Irr(\mathbf{T}_2^F)$, then the equality
$R_{\mathbf{T}_1}^\mathbf{G}(\theta_1)=R_{\mathbf{T}_2}^\mathbf{G}(\theta_2)$
holds if and only if there exists some $x \in \mathbf{G}^F$ such
that~$\mathbf{T}_1^{\, x} = \mathbf{T}_2$
and~${\theta_1^{\, x} = \theta_2}$.  Moreover,
$\mathbf{G}^F$-conjugacy classes of pairs $(\mathbf{T},\theta)$
correspond bijectively to $(\mathbf{G}^*)^{F^*}\!$-conjugacy classes
of pairs $(\mathbf{T}^*,g)$ for the dual
group~$(\mathbf{G}^*)^{F^*}\!$, where $\mathbf{T}^*$ is an
$F^*$-stable maximal torus and $g$ is a (semisimple)
  element of $(\mathbf{T}^*)^{F^*}\!$; see
\cite[Proposition~11.1.16]{DM20}.
% \footnote{PERHAPS refer to \cite[Remark~2.3.21]{GM20} or
% \cite[Corollary~2.5.14]{GM20} instead?}
Accordingly, we write
\[
  R_{\mathbf{T}^*}^\mathbf{G}(g) = R_\mathbf{T}^\mathbf{G}(\theta)
\]
in this situation.

The rational series $\mathcal{E}(\mathbf{G}^F,g)$ of irreducible
characters of $\mathbf{G}^F$ associated with the semisimple
element~$g$ is the set of irreducible characters of $\mathbf{G}^F$
which occur in some Deligne--Lusztig
character~$R_{\mathbf{T}^*}^{\mathbf{G}}(g)$, for some
$F^*$-stable maximal torus $\mathbf{T}^*$ such that
$g \in (\mathbf{T}^*)^{F^*}\!$.  The series
$\mathcal{E}(\mathbf{G}^F,1)$ corresponding to the trivial element is
the set of unipotent characters of $\mathbf{G}^F\!$.  This set-up
produces a partition
\begin{equation} \label{equ:DL-Irr-partition}
  \Irr(\mathbf{G}^F) = \bigsqcup_{g \in \mathcal{G}_\mathrm{s}}
  \mathcal{E}({\mathbf{G}}^{F},g),
\end{equation}
where $\mathcal{G}_\mathrm{s}$ denotes a set of representatives for
the conjugacy classes of semisimple elements in
$(\mathbf{G}^*)^{F^*}\!\!$; compare with \cite[Theorem~2.6.2]{GM20}.

The characters occurring in a non-unipotent rational series
$\mathcal{E}({\mathbf{G}}^{F},g)$ are related to the unipotent
characters of the centraliser $\mathrm{C}_{\mathbf{G}^*}(g)^{F^*}\!$,
which is a typically smaller reductive group.  If the
centre~$\mathrm{Z}(\mathbf{G})$ is not connected, then the
centraliser~$\mathrm{C}_{\mathbf{G}^*}(g)$ may not be connected.  In
order to incorporate this more complicated situation into our
framework, we denote
by~$\mathcal{E}(\mathrm{C}_{\mathbf{G}^*}(g)^{F^*}\!,1)$ the set of
irreducible constituents of the induced unipotent characters of the
connected part of the centraliser; compare
with~\cite[Remark~2.6.26]{GM20}.  For $(\mathbf{G},F)$ and
$(\mathbf{G}^*,F^*)$ in duality, the cardinality of the group
$\mathbf{G}^F$ equals that of~$(\mathbf{G}^*)^{F^*}\!$; see
\cite[Example~1.6.19]{GM20}.

With these preparations we can now state
Lusztig's `Jordan decomposition of characters'.  For every semisimple
element~$g \in (\mathbf{G}^*)^{F^*}\!$, there is a bijection
\[
  \psi_g \colon \mathcal{E}(\mathbf{G}^F,g) \to
  \mathcal{E}(\mathrm{C}_{\mathbf{G}^*}(g)^{F^*}\!,1)
\]
such that
\begin{equation}
  \label{equ:degree}
  \chi(1) =
  \bigl\vert (\mathbf{G}^*)^{F^*}\! :
  \mathrm{C}_{\mathbf{G}^*}(g)^{F^*} \bigr\vert_{p'} \,\, \psi_g(\chi)(1)
  \qquad \text{for $\chi \in \mathcal{E}(\mathbf{G}^F,g) $},
\end{equation}
where $|\cdot|_{p'}$ denotes the $p$-prime part of a natural number;
see \cite[Theorem~11.5.1 and
Proposition~11.5.6]{DM20}. %\footnote{DELETE alternative reference\cite[Theorem~2.6.22 and Remark~2.6.26]{GM20}?}
The study of
unipotent characters of~$\mathbf{G}^F$ reduces further to the case
when $\mathbf{G}$ is simple of adjoint type; see \cite[Proposition
2.3.15 and Remark~4.2.1]{GM20}.  A detailed description of unipotent
characters can be found in \cite[Chapter~4]{GM20}; in particular, the
degrees of all irreducible unipotent characters are known.

We are interested in the irreducible representations of finite simple
groups of Lie type which arise as
$\mathbf{G}^F/ \mathrm{Z}(\mathbf{G}^F)$ from~\cref{thm:tits}.  For
such a group, inflation yields a natural bijection
\begin{equation}
  \label{equ:irr-G/Z}
  \Irr \bigl( \mathbf{G}^F/ \mathrm{Z}(\mathbf{G}^F) \bigr) \overset{\mathrm{bij.}}{\longrightarrow} \left\{\chi\in \Irr
    (\mathbf{G}^F) \mid \mathrm{Z}(\mathbf{G}^F) \subseteq \ker \chi \right\}.
\end{equation}
For every $\chi \in \Irr(\mathbf{G}^F)$ there exists a pair
$(\mathbf{T},\theta)$ such that
$\langle R_\mathbf{T}^\mathbf{G}(\theta),\chi\rangle\neq 0$.
Moreover, the character values of $\chi$ on elements of the centre
$\mathrm{Z}(\mathbf{G})^F$ are completely determined by the values of
the corresponding linear character~$\theta$ of the torus
  $\mathbf{T}^F$; compare with \cite[Proposition~2.2.20]{GM20}.  This
is the starting point for obtaining the following proposition, which
is a consequence of the proof of \cite[Lemma~4.4]{NT13} and
\cite[Remark~4.6]{NT13}.

\begin{proposition}
  \label{pro:repr-for-simple-groups}
  Let $\mathbf{G}$ be a connected, simply connected simple algebraic
  group equipped with a Frobenius endomorphism~$F$ such that
  $\mathbf{G}^F / \mathrm{Z}(\mathbf{G}^F)$ is a finite simple group
  of Lie type.  Let $\chi \in \Irr(\mathbf{G}^F)$ be contained in the
  rational series $\mathcal{E}(\mathbf{G}^F,g)$ associated to a
  semisimple element $g \in (\mathbf{G}^*)^{F^*}\!$.

  Then $\mathrm{Z}(\mathbf{G}^F) \subseteq\ker \chi$ holds if and only
  if $g \in \bigl[ (\mathbf{G}^*)^{F^*}\!, (\mathbf{G}^*)^{F^*} \bigr]$.
\end{proposition}

%%%
\subsection{A variant of Theorem~\ref{thm:approximation}}

The proof of \cref{thm:approximation} uses the Deligne--Lusztig
decomposition of characters for finite groups of Lie type, as sketched
above.  Our aim is to use \cref{pro:repr-for-simple-groups} in order
to arrive at a version of \cref{thm:approximation} that applies to
finite simple groups of Lie type.  We follow the framework developed
in \cite[Section~3]{AKOV16}.  However, we employ \emph{rational}
Lusztig series instead of \emph{geometric} ones; this by-passes a
detour that was taken in~\cite{AKOV16} due to a misprint in the first
edition of~\cite{DM20}.  Additionally, we provide more details in some
critical passages concerning certain enumerative bounds.  Finally, the
more complicated setting considered here necessitates some supplements
to the arguments familiar from~\cite{AKOV16}.

The notion of a Lie type was described in \cref{def:lie-type-basic}.
Further we recall that, for an irreducible root system~$\Phi$, we
write $\boldsymbol{\Lambda}_\Phi$ to denote the collection of Lie
types $\lambda$ with~$\Phi_\lambda \cong \Phi$ and $\mathcal{Q}$ to
denote the set of all prime powers.

\begin{theorem} \label{thm:Approximation-simple-groups} Let $\Phi$ be
  an irreducible root system.  Then there exist a constant $D \in \R$,
  a finite set $\boldsymbol{d}(\Phi)\in \mathcal{A}^+$, and subsets
  $\boldsymbol{d}(\lambda,q) \subseteq \boldsymbol{d}(\Phi)$ for
  $(\lambda,q)\in \boldsymbol{\Lambda}_\Phi\times \mathcal{Q}$ whose
  union is $\boldsymbol{d}(\Phi)$ such that the following holds.
  
  For every $q \in \mathcal{Q}_{>3}$ and every finite simple group
  $\mathbf{G}^F / \mathrm{Z}(\mathbf{G}^F)$ of Lie type
  $\lambda \in \boldsymbol{\Lambda}_\Phi$, where $\mathbf{G}$ is a connected,
  simply connected simple $\overline{\mathbb{F}_q}$-algebraic group
  and $F$ is a Frobenius endomorphism defining an
  $\mathbb{F}_q$-structure, we have
  \begin{equation} \label{equ:AKOV-variant-claim1}
    \zeta_{\mathbf{G}^F/\mathrm{Z}(\mathbf{G}^F)} - 1 \, \sim_D \,
    \xi_{\boldsymbol{d}(\lambda, q),q}.
  \end{equation}
  Moreover, $(\rk(\Phi), \lvert \Phi^+ \rvert) \in \boldsymbol{d}(\Phi)$ and
  \begin{equation} \label{equ:AKOV-variant-claim2}
    \zeta_{\mathbf{G}^F/\mathrm{Z}(\mathbf{G}^F)} \, \sim_D \, 1 +
    q^{\rk(\Phi) - \lvert \Phi^+ \rvert s}.
  \end{equation}
\end{theorem}

\begin{proof}
  We analyse the distribution of irreducible characters of the finite
  simple groups $\mathbf{G}^F / \mathrm{Z}(\mathbf{G}^F)$ of Lie type,
  working uniformly across the underlying parameters
  $(\lambda,q) \in \boldsymbol{\Lambda}_\Phi \times \mathcal{Q}_{>3}$.
  From \eqref{equ:DL-Irr-partition}, \eqref{equ:irr-G/Z} and
  \cref{pro:repr-for-simple-groups} we deduce that
  \begin{equation}
    \label{equ:zeta-quotient}
    \zeta_{\mathbf{G}^F / \mathrm{Z}(\mathbf{G}^F)}(s) = \sum_{g \in
      \tilde{\mathcal{G}}_\mathrm{s}} \; \sum_{\chi \in
      \mathcal{E}(\mathbf{G}^F,g)} \chi(1)^{-s},
  \end{equation}
  where $\tilde{\mathcal{G}}_\mathrm{s}$ is a set of representatives
  of $(\mathbf{G}^*)^{F^*}\!$-conjugacy classes of semisimple elements
  $g \in (\mathbf{G}^*)^{F^*}\!$ such that
  ${g \in \big[ (\mathbf{G}^*)^{F^*}\!, (\mathbf{G}^*)^{F^*} \big]}$
  and, in particular,
    ${1\in \tilde{\mathcal{G}}_\mathrm{s}}$.  Since unipotent
  characters always restrict trivially to $\mathrm{Z}(\mathbf{G}^F)$,
  we obtain
  \begin{equation}
    \label{equ:decomp-zeta-into-unip-nu}
    \zeta_{\mathbf{G}^F/\mathrm{Z}(\mathbf{G}^F)}(s) =
    \zeta^{\mathrm{u}}_{\mathbf{G}^F}(s)
    + \zeta^{\mathrm{nu}}_{\mathbf{G}^F/\mathrm{Z}(\mathbf{G}^F)}(s),
  \end{equation}
  where the notation
  \[
    \zeta_{\mathbf{G}^F}^{\mathrm{u}}(s) = \sum_{\chi\in
      \mathcal{E}(\mathbf{G}^F,1)} \chi(1)^{-s} \quad \text{and} \quad
    \zeta^{\mathrm{nu}}_{\mathbf{G}^F / \mathrm{Z}(\mathbf{G}^F)}(s) =
    \sum_{g \in \tilde{\mathcal{G}}_\mathrm{s}\smallsetminus \{1\}} \;
    \sum_{\chi \in \mathcal{E}(\mathbf{G}^F,g)} \chi(1)^{-s}
  \]
  is used to separate the contributions from unipotent and
  non-unipotent characters.
  
  Information about unipotent characters is more readily available.
  By \cite[Proposition~3.5]{AKOV16}, there exist $D_0 \in \R$, only
  depending on~$\Phi$, and finite sets
  $\boldsymbol{b}(\lambda,q) \subseteq \boldsymbol{b}(\Phi) \in
  \mathcal{A}^+$ such that
  \begin{equation} \label{equ:approx-unip}
    \zeta^{\mathrm{u}}_{\mathbf{G}^F}(s) - 1 \, \sim_{D_0} \,
    \xi_{\boldsymbol{b}(\lambda,q)}(s) \qquad \text{and} \qquad
    \zeta_{\mathbf{G}^F}^{\mathrm{u}}(s) \, \sim_{D_0} \, 1.
  \end{equation}
  
  In order to deal with the non-unipotent characters, we use
  \eqref{equ:degree} to obtain
  \[
    \zeta^{\mathrm{nu}}_{\mathbf{G}^F / \mathrm{Z}(\mathbf{G}^F)}(s) =
    \sum_{g \in \tilde{\mathcal{G}}_\mathrm{s} \smallsetminus \{1\}}
    \bigl| (\mathbf{G}^*)^{F^*} \! : \mathrm{C}_{\mathbf{G}^*}(g)^{F^*}
    \bigr|_{p'}^{\, -s} \cdot
    \zeta^{\mathrm{u}}_{\mathrm{C}_{\mathbf{G}^*}(g)^{F^*}}(s),
  \]
  where $\zeta^{\mathrm{u}}_{\mathrm{C}_{\mathbf{G}^*}(g)^{F^*}}(s)$
  collects the contribution from the unipotent characters of the
  reductive group~$\mathrm{C}_{\mathbf{G}^*}(g)^{F^*}\!$, i.e., the
  irreducible constituents of the induced unipotent characters of the
  connected part $\mathrm{C}^\circ_{\mathbf{G}^*}(g)^{F^*}\!$.  The
  index
  $\lvert \mathrm{C}_{\mathbf{G}^*}(g) :
  \mathrm{C}^\circ_{\mathbf{G}^*}(g) \rvert$ is bounded by a constant
  $D_1\in \R$, which depends only on $\Phi$; compare
  with~\cite[Theorem~2.2.14]{GM20}.  Hence
  \cite[Proposition~3.5]{AKOV16}, the analogue of
  \eqref{equ:approx-unip} for
  $\mathrm{C}^\circ_{\mathbf{G}^*}(g_s)^{F^*}\!$, shows that there is
  a constant $D_2\in \R$, which depends only on~$\Phi$, such that
  \[
    \zeta^{\mathrm{nu}}_{\mathbf{G}^F / \mathrm{Z}(\mathbf{G}^F)}(s)
    \, \sim_{D_ 1 D_2} \, \sum_{g \in
      \tilde{\mathcal{G}}_\mathrm{s}\smallsetminus\{1\}} \bigl|
    (\mathbf{G}^*)^{F^*}\! : \mathrm{C}^\circ_{\mathbf{G}^*}(g)^{F^*}
    \bigr|_{p'}^{\, -s}.
  \]

  By \cite[Proposition~26.6]{MT11}, every semisimple element
  $g \in (\mathbf{G}^*)^{F^*}\!$ lies in an $F^*$-stable maximal torus
  $\mathbf{T}$ of~$\mathbf{G}^*$.  By \cite[Proposition~25.1]{MT11},
  the number of $(\mathbf{G}^*)^{F^*}\!$-conjugacy classes of
  $F^*$-stable maximal tori is finite and bounded above in terms
  of~$\Phi$ only.  Moreover, for each $F^*$-stable maximal torus
  $\mathbf{T}$, the number of connected centralisers
  $\mathrm{C}^\circ_{\mathbf{G}^*}(g)$ of elements $g \in \mathbf{T}$
  is bounded above in terms of $\Phi$; see
  \cite[Proposition~14.2]{MT11}.  Thus we arrive at a finite
  collection of reductive algebraic subgroups~
  $\mathbf{K}_1,\dots,\mathbf{K}_N$ of~$\mathbf{G}^*$, with associated
  root systems $\Psi_1,\dots,\Psi_N$, that form a set of
  representatives for the $(\mathbf{G}^*)^{F^*}\!$-conjugacy classes
  of connected centralisers $\mathrm{C}^\circ_{\mathbf{G}^*}(g)$ for
  non-trivial semisimple
  elements~$g \in \bigl[ (\mathbf{G}^*)^{F^*}\!, (\mathbf{G}^*)^{F^*}
  \bigr]$, where $N$ is bounded above in terms of $\Phi$ only.  This
  induces a partition
  $\tilde{\mathcal{G}_\mathrm{s}}\setminus \{1\} = \bigsqcup_{i=1}^N
  \tilde {\mathcal{G}_i}$, where $\tilde{\mathcal{G}_i}$ consists of
  all $g \in \tilde{\mathcal{G}_\mathrm{s}}$ such that
  $\mathrm{C}^\circ_{\mathbf{G}^*}(g)$ is
  $(\mathbf{G}^*)^{F^*}\!$-conjugate to $\mathbf{K}_i$, and by
  adapting $\tilde{\mathcal{G}}_\mathrm{s}$, we may assume that
  $\mathrm{C}^\circ_{\mathbf{G}^*}(g) = \mathbf{K}_i$ for all
  $g \in \tilde{\mathcal{G}_i}$.
  
  Using~\cite[Corollary~24.6]{MT11} there is a constant $D_3 \in \R$,
  which depends only on~$\Phi$, such that, for each
  $i \in \{1,\ldots,N\}$ and $g \in \tilde{\mathcal{G}_i}$,
  \begin{equation} \label{equ:estimate-order-positive-roots}
    \bigl| (\mathbf{G}^*)^{F^*}\! :
    \mathrm{C}^\circ_{\mathbf{G}^*}(g)^{F^*}
    \bigr|_{p'} \, \sim_{D_3} \, q^{\dim \mathbf{G}-|\Phi^+|-\dim
      \mathbf{K}_i+|\Psi_i^+|} = q^{|\Phi^+|-|\Psi_i^+|}.
  \end{equation}
  We remark that in \eqref{equ:estimate-order-positive-roots} the
  complex variable~$s$ does not appear; the assertion holds, because
  the two quantities bound each other up to a constant factor
  uniformly across all possible~$q$.  Setting $D_4 = D_1D_2D_3$, we
  have arrived at
  \begin{equation}
    \label{equ:D4}
    \zeta^{\mathrm{nu}}_{\mathbf{G}^F / \mathrm{Z}(\mathbf{G}^F)}(s)
    \, \sim_{D_4} \,
    \sum\nolimits_{i=1}^N\,\, \lvert \tilde{\mathcal{G}_i} \rvert
    \,\,q^{-(|\Phi^+|-|\Psi_i^+|)s}.
  \end{equation}

  For the following discussion we temporarily fix
  $i \in \{1,\dots,N\}$.  We set
  \[
    \mathcal{K}_i = \{ k \in \bigl[ (\mathbf{G}^*)^{F^*}\!,
    (\mathbf{G}^*)^{F^*} \bigr] \mid k \text{ semisimple and }
    \mathrm{C}_{\mathbf{G}^*}^\circ(k) = \mathbf{K}_i \bigr\}
  \]
  and observe that for each $k \in \mathcal{K}_i$ there are
  $g \in \tilde{\mathcal{G}_i}$ and
  $x \in \mathrm{N}_{(\mathbf{G}^*)^{F^*}} (\mathbf{K}_i)$ such that
  $k = g^x$.  From~\cite[Lemma~2.2(ii)]{LS05} we deduce that the
  number of distinct conjugates of this form is bounded by a constant
  $D_5 \in \R$ that depends only on~$\Phi$.  This yields
  \begin{equation}
    \label{equ:D5}
    \lvert \tilde{\mathcal{G}_i} \rvert \, \sim_{D_5} \, \lvert
    \mathcal{K}_i \rvert.
  \end{equation}
  Let $\mathbf{G}^*_\mathrm{sc}$ be the simply connected cover
  of~$\mathbf{G}^*$ and
  $\pi \colon \mathbf{G}_\mathrm{sc}^* \to \mathbf{G}^*$ an isogeny
  with central kernel.  From \cite[Proposition~24.21]{MT11} we recall
  that $\pi$ maps $(\mathbf{G}_{\mathrm{sc}}^*)^{F^*}\!$ onto
  $[(\mathbf{G}^*)^{F^*}\! ,(\mathbf{G}^*)^{F^*}]$, because we are in the
  setting of \cref{thm:tits} with $q>3$.  In this situation we fix an
  element $h_i \in (\mathbf{G}_\mathrm{sc}^*)^{F^*}\! $ such that
  $k_i = h_i \pi \in \mathcal{K}_i$.
  From~\cite[Subsection~1.3.10~(e)]{GM20} we see that~$\pi$
  maps~$\mathrm{C}^\circ_{\mathbf{G}_\mathrm{sc}^*}(h_i)$
  onto~$\mathrm{C}_{\mathbf{G}^*}^\circ(k_i)$.  We
  set~${\mathbf{H}_i =
    \mathrm{C}^\circ_{\mathbf{G}_\mathrm{sc}^*}(h_i)}$ and consider
  the set
  \[
    \mathcal{H}_i = \bigl\{ h \in (\mathbf{G}^*_\mathrm{sc})^{F^*}\!
    \mid h \text{ semisimple } \text{and }
    \mathrm{C}_{\mathbf{G}_\mathrm{sc}^*}^\circ(h) = \mathbf{H}_i \bigr\}.
  \]
  In fact, centralisers of semisimple elements in the simply connected
  group $\mathbf{G}_\mathrm{sc}^*$ are connected, but we prefer to
  write $\mathrm{C}_{\mathbf{G}_\mathrm{sc}^*}^\circ(h)$ to match the
  notation used for~$\mathbf{G}^*$.

  The isogeny $\pi \colon \mathbf{G}_\mathrm{sc}^* \to \mathbf{G}^*$
  induces a map $\mathcal{H}_i \to \mathcal{K}_i$, whose fibres have
  size at most~$\lvert \mathrm{Z}(\mathbf{G}_\mathrm{sc}^*) \rvert$.
  Moreover, the number of conjugacy classes in
  $(\mathbf{G}_\mathrm{sc}^*)^{F^*}\!$ that map to one and the same
  conjugacy class of $(\mathbf{G}^*)^{F^*}\!$ under $\pi$ is at most
  $\lvert \mathrm{Z}(\mathbf{G}_\mathrm{sc}^*) \rvert$.  Thus the
  number of $(\mathbf{G}_\mathrm{sc}^*)^{F^*}\!$-conjugacy classes of
  connected centralisers $\mathrm{C}_{\mathbf{G}_\mathrm{sc}^*}(h)$
  for semisimple elements $h \in (\mathbf{G}_\mathrm{sc}^*)^{F^*}\!$ is
  bounded above in terms of $\Phi$ only, and applying the arguments
  that lead up to \eqref{equ:D5} for $\mathbf{G}_\mathrm{sc}^*$ in
  place of $\mathbf{G}^*$ we deduce that there is a constant $D_6$
  that depends on $\Phi$ only such that
  \begin{equation} \label{equ:D6}
    \lvert \mathcal{K}_i \rvert \,\sim_{D_6} \, \lvert \mathcal{H}_i
    \rvert.
  \end{equation}
  
  Now let $\mathbf{T}_i$ be a maximal torus of $\mathbf{H}_i$, hence
  also a maximal torus in $\mathbf{G}_\mathrm{sc}^*$.  Let~$\Delta_i$
  denote the set of roots of~$\mathbf{H}_i$ and $\Lambda_i $ the set
  of roots of $\mathbf{G}_\mathrm{sc}^*$, with respect
  to~$\mathbf{T}_i$.  We observe that $\Lambda_i$ forms a root system
  isomorphic to~$\Phi$.  For every $\alpha\in\Delta_i$, the identity
  $\alpha(h_i)=1$ holds, and \cite[\S~2.2.13]{GM20} shows that
  \[
    \mathbf{H}_i = \bigl\langle \mathbf{T}_i\cup
    \bigcup\nolimits_{\alpha\in \Delta_i} \mathbf{U}_\alpha
    \bigr\rangle,
  \]
  where $\mathbf{U}_\alpha$ denotes the root subgroup of
  $\mathbf{H}_i $ associated to $\alpha$.
  
  We observe that
  $\mathcal{H}_i \subseteq \mathrm{Z}(\mathbf{H}_i)^{F^*}\!$.  Since
  $\mathbf{H}_i$ is a connected reductive group, its radical is the
  central torus $\mathrm{Z}(\mathbf{H}_i)^\circ$, its derived group
  $[\mathbf{H}_i,\mathbf{H}_i]$ is semisimple, and
  $\mathbf{H}_i=\mathrm{Z}(\mathbf{H}_i)^\circ[\mathbf{H}_i,\mathbf{H}_i]$
  with
  $\lvert
  \mathrm{Z}(\mathbf{H}_i)^\circ\cap[\mathbf{H}_i,\mathbf{H}_i] \rvert
  < \infty$; compare with \cite[Proposition 6.20, Theorem 8.21, and
  Corollary 8.22]{MT11}.  Thus
  $\mathrm{Z}(\mathbf{H}_i) = \mathrm{Z}(\mathbf{H}_i)^\circ \, Z$,
  where $Z = \mathrm{Z}([\mathbf{H}_i,\mathbf{H}_i])$ is finite.  The
  order $\lvert Z \rvert$ and thus the exponent $e$ of $Z$, a divisor
  of $\lvert Z \rvert$, are bounded by a constant that depends only
  on~$\Delta_i$, hence they can be bounded by a constant $C_1 \in \R$
  that depends only on~$\Lambda_i \cong \Phi$.

  Next we distinguish three cases and obtain two further constants
  $ C_2, C_3 \in \R$, each depending only on~$\Phi$, such that
  \begin{equation} \label{equ:mathcal-H-i-bound} \lvert \mathcal{H}_i
    \rvert \,\sim_C\, q^{\dim \mathrm{Z}(\mathbf{H}_i)} \qquad
    \text{for} \quad  C = \max \{C_1, C_2, C_3 \}.
  \end{equation}

  \smallskip
    
  \noindent \textit{Case 1:} $\dim \mathrm{Z}(\mathbf{H}_i) = 0$.  In
  this situation we have
  $\lvert \mathcal{H}_i \rvert \le \lvert Z \rvert$ and
  $\lvert \mathcal{H}_i \rvert \sim_{C_1} 1$.

  \smallskip
  
  \noindent \textit{Case 2:} $\dim \mathrm{Z}(\mathbf{H}_i) \ge 1$ and
  $\mathbf{H}_i$ is a Levi subgroup of~$\mathbf{G}_\mathrm{sc}^*$ so
  that
  $\mathbf{H}_i =
  \mathrm{C}_{\mathbf{G}_\mathrm{sc}^*}(\mathrm{Z}(\mathbf{H}_i)^\circ)$;
  compare with \cite[Proposition~12.6]{MT11}.  From
  $\lvert \mathrm{Z}(\mathbf{H}_i) : \mathrm{Z}(\mathbf{H}_i)^\circ
  \rvert \le \lvert Z \rvert \le C_1$ we deduce that
  \begin{equation}
    \label{equ:connected-part-of-the-center}
    \mathcal{H}_i \,\sim_{C_1}\, \mathcal{T}_{\mathbf{H}_i}, \qquad
    \text{where} \quad 
    \mathcal{T}_{\mathbf{H}_i}=
    \mathrm{Z}(\mathbf{H}_i)^\circ\smallsetminus \{ h \in
    \mathrm{Z}(\mathbf{H}_i)^\circ\mid
    \mathrm{C}^\circ_{\mathbf{G}_\mathrm{sc}^*}(h) \supsetneqq
    \mathbf{H}_i \}.
  \end{equation}
  From \cite[Theorem~8.17]{MT11} we see that
  $\mathrm{Z}(\mathbf{H}_i) = \bigcap_{\alpha \in \Delta_i}\ker
  \alpha$ and furthermore
  \[
    \mathcal{T}_{\mathbf{H}_i}
    % = \Big( \bigcap\nolimits_{\alpha\in
    %   \Delta_i}\ker \alpha \Big)^\circ \;\smallsetminus\; \Big(
    % \bigcup\nolimits_{\beta\in \Lambda_i\smallsetminus\Delta_i} \ker
    % \beta \Big)
    = \mathrm{Z}(\mathbf{H}_i)^\circ \;\smallsetminus\;
    \bigcup\nolimits_{\beta\in \Lambda_i\smallsetminus\Delta_i} \bigl(
    \mathrm{Z}(\mathbf{H}_i) \cap \ker \beta \bigr).
  \]
  If we had
  $\dim(\mathrm{Z}(\mathbf{H}_i) \cap \ker \beta)=\dim
  (\mathrm{Z}(\mathbf{H}_i))$ for some
  $\beta\in\Lambda_i \smallsetminus \Delta_i$ it would follow that
  $\mathrm{Z}(\mathbf{H}_i)^\circ\subseteq \ker\beta$, and hence
  $\mathbf{U}_\beta\subseteq
  \mathrm{C}_{\mathbf{G}_\mathrm{sc}^*}(\mathrm{Z}(\mathbf{H}_i)^\circ)
  = \mathbf{H}_i$,
  in contradiction to our initial set-up.  Hence for every
  $\beta \in \Lambda_i\smallsetminus\Delta_i$, the group
  $(\mathrm{Z}(\mathbf{H}_i) \cap \ker \beta)^\circ$ is a proper
  subtorus of~$\mathrm{Z}(\mathbf{H}_i)^\circ$.  Moreover,
  $\lvert (\mathrm{Z}(\mathbf{H}_i) \cap \ker \beta) :
  (\mathrm{Z}(\mathbf{H}_i) \cap \ker \beta)^\circ \rvert$, viz.\ the
  number of irreducible components of
  $\mathrm{Z}(\mathbf{H}_i) \cap \ker \beta$, is bounded by some
  constant depending only on $\Phi$; see, for instance,
  \cite[Corollary~9.7.9]{Gro67}.  A basic bound for the number of
  $F^*$-fixed points of the torus~$\mathrm{Z}(\mathbf{H}_i)^\circ$ is
  \begin{equation} \label{equ:asymp-for-torus}
    (q-1)^{\dim \mathrm{Z}(\mathbf{H}_i)} \le \lvert
    (\mathrm{Z}(\mathbf{H}_i)^\circ)^{F^*}\!  \rvert \le
    (q+1)^{\dim \mathrm{Z}(\mathbf{H}_i)};
  \end{equation}
  compare with \cite[Lemma~1.6.6]{GM20}.
  Approximating in a similar way the numbers of $F^*$-fixed points of
  the finitely many proper subtori
  $(\mathrm{Z}(\mathbf{H}_i) \cap \ker \beta)^\circ$, we deduce that
  $\lvert \mathcal{H}_i \rvert \sim_{C_{2}}q^{\dim
    \mathrm{Z}(\mathbf{H}_i)}$ for a constant $C_{2} \in \R$ that
  depends only on~$\Lambda_i \cong \Phi$.

  \smallskip
  
  \noindent \textit{Case 3:} $\dim \mathrm{Z}(\mathbf{H}_i) \ge 1$ and
  $\mathbf{H}_i$ is not a Levi subgroup of $\mathbf{G}_\mathrm{sc}^*$.
  By \cite[Propositions~12.10]{MT11}, the subgroup
  $\mathbf{L}_i =
  \mathrm{C}_{\mathbf{G}_\mathrm{sc}^*}(\mathrm{Z}(\mathbf{H}_i)^\circ)$
  is a Levi subgroup of $\mathbf{G}_\mathrm{sc}^*$ with
  $\mathrm{Z}(\mathbf{L}_i)^\circ = \mathrm{Z}(\mathbf{H}_i)^\circ$
  and, in particular,
  $\dim \mathrm{Z}(\mathbf{L}_i) = \dim \mathrm{Z}(\mathbf{H}_i)$.
  Moreover, we have
  $\mathrm{C}_{\mathbf{G}_\mathrm{sc}^*}^\circ(t) \supseteq
  \mathbf{L}_i\supsetneqq \mathbf{H}_i$ for every
  $t\in \mathrm{Z}(\mathbf{H}_i)^\circ$.  We recall that
  $h_i \in \mathcal{H}_i \subseteq \mathrm{Z}(\mathbf{H}_i) =
  \mathrm{Z}(\mathbf{H}_i)^\circ Z$ and that $e$ denotes the exponent
  of~$Z$.  In particular,
  $t_i = h_i^{\, e} \in (\mathrm{Z}(\mathbf{H}_i)^\circ)^{F^*}\!$.

  From $\mathcal{H}_i \subseteq \mathrm{Z}(\mathbf{H}_i)^\circ Z$, the
  bound $\lvert Z \rvert \le C_1$ and the approximation of
  $\lvert (\mathrm{Z}(\mathbf{H}_i)^\circ)^{F^*} \rvert$ in
  \eqref{equ:asymp-for-torus} we obtain a correct asymptotic upper
  bound for $\lvert \mathcal{H}_i \rvert$ and we need to establish the
  corresponding lower bound.  We claim that
  \[
    \bigl\{ t h_i \mid t \in (\mathrm{Z}(\mathbf{H}_i)^\circ)^{F^*}\!
    \text{ with } \mathrm{C}_{\mathbf{G}_\mathrm{sc}^*}^\circ(t^e t_i)
    = \mathbf{L}_i \bigr\} \subseteq \mathcal{H}_i.
  \]
  Indeed, if $h = t h_i$ for $t \in \mathrm{Z}(\mathbf{H}_i)^\circ$
  with~$\mathrm{C}_{\mathbf{G}_\mathrm{sc}^*}^\circ(t^e t_i) =
  \mathbf{L}_i$, we deduce that
  \begin{multline*}
    \mathbf{H}_i \subseteq
    \mathrm{C}_{\mathbf{G}_\mathrm{sc}^*}^\circ(h) =
    \mathrm{C}_{\mathbf{G}_\mathrm{sc}^*}^\circ \bigl( h^e \bigr) \cap
    \mathrm{C}_{\mathbf{G}_\mathrm{sc}^*}^\circ(h) =
    \mathrm{C}_{\mathbf{G}_\mathrm{sc}^*}^\circ(t^e t_i) \cap
    \mathrm{C}_{\mathbf{G}_\mathrm{sc}^*}^\circ(t h_i) \\
    = \mathbf{L}_i\cap \mathrm{C}_{\mathbf{G}_\mathrm{sc}^*}^\circ(t
    h_i) = \mathrm{C}_{\mathbf{L}_i}^\circ(h_i) = \mathbf{H}_i,
    \qquad \text{hence} \quad
    \mathrm{C}_{\mathbf{G}_\mathrm{sc}^*}^\circ(h) = \mathbf{H}_i.
  \end{multline*}
  Hence it suffices to see that there is a constant $C_4 \in \R$
  depending only on $\Phi$ such that
  \begin{equation} \label{equ:C4-asymp}
    \bigl\vert \bigl\{ t \in (\mathrm{Z}(\mathbf{H}_i)^\circ)^{F^*}\! \mid
    \mathrm{C}_{\mathbf{G}_\mathrm{sc}^*}^\circ(t^e t_i) =
    \mathbf{L}_i \bigr\} \bigr\vert \,\sim_{C_4}\, q^{\dim
      \mathrm{Z}(\mathbf{H}_i)}.
  \end{equation}
  As the fibres of the map
  $(\mathrm{Z}(\mathbf{H}_i)^\circ)^{F^*}\!\! \to
  (\mathrm{Z}(\mathbf{H}_i)^\circ)^{F^*}\!$, $t \mapsto t^e t_i$ have
  size at most $e^{\dim \mathrm{Z}(\mathbf{L}_i)}$, the
  approximation~\eqref{equ:C4-asymp} follows from
  $\mathrm{Z}(\mathbf{L}_i)^\circ = \mathrm{Z}(\mathbf{H}_i)^\circ$,
  the elementary bound
  \[
    e^{-\dim \mathrm{Z}(\mathbf{L}_i)} \bigl\vert
    (\mathrm{Z}(\mathbf{L}_i)^\circ)^{F^*} \bigr\vert \le \bigl\vert
    \bigl\{ t^e t_i \mid t \in (\mathrm{Z}(\mathbf{L}_i)^\circ) ^{F^*}
    \bigr\}
    \bigr\vert \le \bigl\vert
    (\mathrm{Z}(\mathbf{L}_i)^\circ)^{F^*} \bigr\vert
  \]
  and an approximation, in a similar way as in Case~2, of
  the $F^*$-fixed points of
  \begin{equation}
    \mathcal{T}_{\mathbf{L}_i,t_i} = \{ t^e t_i \mid t \in 
    \mathrm{Z}(\mathbf{L}_i)^\circ \} \smallsetminus
    \bigcup\nolimits_{\beta\in \Lambda_i\smallsetminus \Sigma_i}
    \bigl( \mathrm{Z}(\mathbf{L}_i) \cap \ker \beta \bigr),
  \end{equation}
  where $\Sigma_i$ denotes the set of roots of $\mathbf{L}_i$ with
  respect to~$\mathbf{T}_i$.  Hence we deduce that
  $\lvert \mathcal{H}_i \rvert \,\sim_{C_3}\, q^{\dim
    \mathrm{Z}(\mathbf{H}_i)}$ for a constant $C_3 \in \R$ that
  depends only on $\Phi$.
  
  \smallskip

  Thus we have established~\eqref{equ:mathcal-H-i-bound}.  Using
  \eqref{equ:D4}, \eqref{equ:D5}, \eqref{equ:D6} and
  \eqref{equ:mathcal-H-i-bound}, and setting
  $D_7 = D_4 \, D_5 \, D_6 \, C $, we obtain
  \[
    \zeta^{\mathrm{nu}}_{\mathbf{G}^F/\mathrm{Z}(\mathbf{G}^F)}(s) \,
    \sim_{D_7} \, \sum\nolimits_{i=1}^N q^{\dim
      \mathrm{Z}(\mathbf{K}_i)-(|\Phi^+|-|\Psi_i^+|)s}.
  \]
  The analysis at the very end of \cite[Proof of
  Theorem~5.1]{LaLu08} allows us to deduce that
  \begin{equation}
    \label{equ:help-for-newton}
    \frac{\dim
      \mathrm{Z}(\mathbf{K}_i)}{|\Phi^+|-|\Psi^+|}\le \frac{\rk \Phi}{|\Phi^+|};
  \end{equation}
  also compare with \cite[Lemma~2.5]{LS04}.
    
  From \cite[Lemma 2.3.11]{GM20} we see that for all~$q$ exceeding a
  lower bound which depends only on $\Phi$, there exist regular
  semisimple elements in~$(\mathbf{G}^*_\mathrm{sc})^{F^*}\!$.  Such
  regular elements are sent via the isogeny
  $\pi \colon \mathbf{G}_\mathrm{sc}^* \to \mathbf{G}^*$ to regular
  elements of $(\mathbf{G}^*)^{F^*}\!$ that lie
  in~$\bigl[ (\mathbf{G}^*)^{F^*}\! , (\mathbf{G}^*)^{F^*} \bigr]$.  Thus
  there exists $i\in \{1,\dots,N\}$ such that
  \begin{equation}
    \label{equ:maximal-reached}
    (\dim \mathrm{Z}(\mathbf{K}_i), |\Phi^+|-|\Psi_i^+| ) = (\rk \Phi, |\Phi ^+|).
  \end{equation}
  Recall that the simple group $\mathbf{G}^F/\mathrm{Z}(\mathbf{G}^F)$
  is characterised up to isomorphism by its Lie type $\lambda$ and the
  parameter~$q$.  We set
  \[
    \boldsymbol{a}^*(\lambda,q) = \{(\dim \mathrm{Z}(\mathbf{K}_i),
    \lvert \Phi^+ \rvert - \lvert \Psi_i^+ \rvert )\mid 1\le i\le N\}
    \subseteq \mathcal{A}^+
  \]
  to obtain
  \begin{equation}\label{equ:approx-non-unip}
    \zeta^{\mathrm{nu}}_{\mathbf{G}^F/\mathrm{Z}(\mathbf{G}^F)}(s) \,
    \sim_{D_7 N} \, \sum_{(m,n)\in \boldsymbol{a}^*(\lambda,q)} q^{m-ns},
  \end{equation} 
  where $N$ is added as a factor because different $\mathbf{K}_i$ and
  $\mathbf{K}_j$, $i\neq j$, could contribute the same element
  to~$\boldsymbol{a}^*(\lambda,q)$.

  Using the decomposition \eqref{equ:decomp-zeta-into-unip-nu} and the
  estimates \eqref{equ:approx-unip} and \eqref{equ:approx-non-unip},
  we produce a constant $D\in \R$ such that
  \[
    \zeta_{\mathbf{G}^F/\mathrm{Z}(\mathbf{G}^F)}-1\,\sim_D\,
    \xi_{\boldsymbol{d}(\lambda,q), q}, \quad \text{where
      $\boldsymbol{d}(\lambda,q) = \boldsymbol{b}(\lambda,q) \cup
      \boldsymbol{a}^*(\lambda,q)$.}
  \]
  The union $\boldsymbol{d}(\Phi)$ of all the sets
  $\boldsymbol{d}(\lambda,q)$ arising in the above manner is finite,
  and the proof of \eqref{equ:AKOV-variant-claim1} is complete.

  It remains to justify \eqref{equ:AKOV-variant-claim2}, subject to a
  possibly required additional increment of~$D$.  First we observe
  from \eqref{equ:AKOV-variant-claim1} and~\eqref{equ:approx-unip}
  that, subject to $D \ge \max \{1, D_0^{\, 2} \}$,
  \[
    \zeta_{\mathbf{G}^F/\mathrm{Z}(\mathbf{G}^F)} = 1 +
    (\zeta_{\mathbf{G}^F/\mathrm{Z}(\mathbf{G}^F)} - 1) \,\sim_D\, 1
    +\xi_{\boldsymbol{d}(\lambda,q), q} \,\sim_D\, 1
    +\xi_{\boldsymbol{a}^*(\lambda,q), q}.
  \]
  As explained in \cite[Remark~2.7]{AKOV16}, it is pertinent to
  compare the ``north-west''-Newton polytopes associated to
  $\{(0,0)\} \cup \boldsymbol{a}^*(\lambda,q)$ and
  $\{(0,0), (\rk \Phi, \lvert \Phi ^+ \rvert)\}$.  Our construction of
  $\boldsymbol{a}^*(\lambda,q)$ ensures that
  $(\rk \Phi, \lvert \Phi ^+ \rvert) \in \boldsymbol{a}^*(\lambda,q)$
  for all sufficiently large~$q$ and that
  \[
    m \le \rk(\Phi), \quad n \le \lvert \Phi^+ \rvert, \quad m/n \le
    \rk(\Phi) / \lvert \Phi^+ \rvert\quad \text{for all
      $(m,n) \in \boldsymbol{a}^*(\lambda,q)$.}
  \]
  This shows that the two relevant Newton polytopes are the same for
  all sufficiently large~$q$. By \cite[Remark~2.7]{AKOV16}, a suitable
  increment of $D$ leads to~\eqref{equ:AKOV-variant-claim2}.
\end{proof}

\begin{remark}
  In the proof of \cref{thm:Approximation-simple-groups} we treat
  explicitly the case that the connected centraliser of a semisimple
  element is not a Levi subgroup, a situation which occurs, for
  instance, in symplectic groups.  It seems that in \cite[Proof of
  Lemma~3.13]{AKOV16} this possibility was simply overlooked.  Our
  analysis can be used in essentially the same way to fill this
  potential gap.
\end{remark}

%\bibliography{Mylibrary} 

\begin{thebibliography}{99}

\bibitem{AiAv18}  A.~Aizenbud and N.~Avni.
  \newblock Counting points of schemes over finite rings and counting
  representations of arithmetic lattices.
  \newblock {\em Duke Math.\ J.}, 167(14):2721--2743, 2018.

\bibitem{AKOV16} N.~Avni, B.~Klopsch, U.~Onn, and C.~Voll.
  \newblock Arithmetic groups, base change, and representation growth.
  \newblock {\em Geom.\ Funct. Anal.}, 26(1):67--135, 2016.

\bibitem{AKOV16b} N.~Avni, B.~Klopsch, U.~Onn, and C.~Voll.
  \newblock Similarity classes of integral $\mathfrak{p}$-adic matrices and
  representation zeta functions of groups of type $\mathsf{A}_2$.
  \newblock {\em Proc.\ Lond.\ Math.\ Soc.}, 112(2):267--350, 2016.

% \bibitem{BLMM02} H.~Bass, A.~Lubotzky, A.~R. Magid, and S.~Mozes.
%   \newblock The proalgebraic completion of rigid groups.
%   \newblock {\em Geom.\ Dedicata}, 95(1):19--58, 2002.

% \bibitem{Car81} R.~W. Carter.
%   \newblock Centralizers of semisimple elements in the finite
%   classical groups.
%   \newblock {\em Proc.\ Lond.\ Math.\ Soc.}, 42(1):1--41, 1981.

\bibitem{Car93} R.~W. Carter.
  \newblock {\em Finite groups of {{Lie}} type: conjugacy classes and
    complex characters}.
  \newblock Wiley Classics Library. Chichester; New York, 1993.

\bibitem{Co78} B.~N.~Cooperstein.
  \newblock Minimal degree for a permutation representation of a
  classical group.
  \newblock {\em Israel J.\ Math.}, 30(3):213--235, 1978.

\bibitem {CoKiVa24} 
  G.~Corob Cook, S.~Kionke and M.~Vannacci.
  \newblock Weil zeta functions of group representations over finite
  fields.
  \newblock {\em Selecta Math.\ (N.S.)}, 30(3):Paper No.\ 46,
    57 pp., 2024.
  
% \bibitem{DL76} P.~Deligne and G.~Lusztig.
%   \newblock Representations of reductive groups over finite fields.
%   \newblock {\em Ann.\ of Math.}, 103(1):103--161, 1976.

\bibitem{DeLi85} D.~I.~Deriziotis and M.~W.~Liebeck.
  \newblock Centralizers of semisimple elements in finite twisted
  groups of Lie type.
  \newblock {\em J.\ Lond.\ Math.\ Soc.}, 31(1):48--54, 1985.

\bibitem{DM20} F.~Digne and J.~Michel.
  \newblock {\em Representations of finite groups of {{Lie}} type}.
  \newblock {Cambridge University Press}, 2 edition, 2020.

\bibitem{Di69}  J.~D.~Dixon.
  \newblock The probability of generating the symmetric group.
  \newblock {\em Math.\ Z.}, 110:199--205, 1969.
  
\bibitem{Gar16} J.~{Garc{\'i}a Rodr{\'i}guez}.
  \newblock {\em Representation growth}.
  \newblock PhD thesis, Madrid, 2016.
  \newblock available as: \texttt{arxiv.1612.06178}.

\bibitem{GM20} M.~Geck and G.~Malle.
  \newblock {\em The character theory of finite groups of {{Lie}}
    type: a guided tour}.
\newblock {Cambridge University Press}, 1 edition, 2020.

\bibitem{GLS97} D.~Gorenstein, R.~Lyons, and R.~Solomon.
  \newblock {\em The classification of the finite simple groups,
    {{Number}} 3}, volume 40.3 of {\em Mathematical {{Surveys}} and
    {{Monographs}}}.
  \newblock American Mathematical Society, Providence, Rhode Island,
  1997.

% \bibitem{Gr02} L.~C.~Grove.
%   \newblock {\em Classical groups and geometric algebra}, Grad.\
%   Stud.\ Math., 39.
%   \newblock American Mathematical Society, Providence, Rhode Island,
%   2002.
  
% \bibitem{Gri84} R.~I. Grigorchuk.
%   \newblock Degrees of growth of finitely generated groups and the
%   theory of invariant means.
%   \newblock {\em Izvestiya Akademii Nauk SSSR. Seriya
%     Matematicheskaya}, 48(5):939--985, 1984.

\bibitem{Gro67} A.~Grothendieck.
  \newblock \'{E}l\'{e}ments de g\'{e}om\'{e}trie
  alg\'{e}brique. {IV}. \'{E}tude locale des sch\'{e}mas et des
  morphismes de sch\'{e}mas {IV}.
  \newblock {\em Inst.\ Hautes {\'E}tudes Sci.\ Publ.\ Math.}, 32:361, 1967.

\bibitem{KaLu90} W.~M.~Kantor and A.~Lubotzky.
  \newblock The probability of generating a finite classical group.
  \newblock {\em Geom.\ Dedicata}, 36(1):67--87, 1990.
  
\bibitem{KN06} M.~Kassabov and N.~Nikolov.
  \newblock Cartesian products as profinite completions.
  \newblock {\em Int.\ Math.\ Res.\ Not.}, Art.\ ID 72947, 17 pp., 2006.

\bibitem{KW91} P.~B. Kleidman and D.~B. Wales.
  \newblock The projective characters of the symmetric groups that
  remain irreducible on subgroups.
  \newblock {\em J.\ Algebra}, 138(2):440--478, 1991.

\bibitem{LaLu08} M.~Larsen and A.~Lubotzky.
  \newblock Representation growth of linear groups.
  \newblock {\em J.\ Eur.\ Math.\ Soc.}, 10: 351--390, 2008.
  
\bibitem{LS04} M.~W. Liebeck and A.~Shalev.
  \newblock Fuchsian groups, coverings of {{Riemann}} surfaces,
  subgroup growth, random quotients and random walks.
  \newblock {\em J.\ Algebra}, 276(2):552--601, 2004.

\bibitem{LS05} M.~W. Liebeck and A.~Shalev.
  \newblock Character degrees and random walks in finite groups of
  {{Lie}} type.
  \newblock {\em Proc.\ Lond.\ Math.\ Soc.}, 90(01):61--86, 2005.

\bibitem{LM04} A.~Lubotzky and B.~Martin.
  \newblock Polynomial representation growth and the congruence
  subgroup problem.
  \newblock {\em Israel J.\ Math.}, 144(2):293--316, 2004.

\bibitem{LPS96} A.~Lubotzky, L.~Pyber, and A.~Shalev.
  \newblock Discrete groups of slow subgroup growth.
  \newblock {\em Israel J.\ Math.}, 96(2):399--418, 1996.

\bibitem{LS03} A.~Lubotzky and D.~Segal.
  \newblock {\em Subgroup growth}.
  \newblock {Birkh{\"a}user Basel}, {Basel}, 2003.

% \bibitem{Lus75} G.~Lusztig.
% \newblock Sur la conjecture de {M}acdonald.
% \newblock {\em C.\ R.\ Acad.\ Sci.\ Paris S{\'e}r.\ A-B},
% 280(6):A317--A320, 1975.

% \bibitem{Lus78} G.~Lusztig.
%   \newblock {\em Representations of finite {{Chevalley}} groups}.
%   \newblock CBMS Regional Conf.\ Ser.\ in Math., 39 American
%   Mathematical Society, Providence, RI, 1978.

\bibitem{MT11} G.~Malle and D.~Testerman.
  \newblock {\em Linear algebraic groups and finite groups of {{Lie}}
    type}.
  \newblock {Cambridge University Press}, 1 edition, 2011.

\bibitem{Nar83} W.~Narkiewicz.
  \newblock {\em Number theory}.
  \newblock World Scientific, Singapore, 1983.

\bibitem{NT13} G.~Navarro and P.~H. Tiep.
  \newblock Characters of relative {${p'}$}-degree over normal
  subgroups.
  \newblock {\em Ann.\ of Math.}, 178(3):1135--1171, 2013.

\bibitem{Pic24} M.~Piccolo.
  \newblock {\em On the representations of quasi-semisimple profinite
    groups, compact $p$-adic analytic groups, and
    {B}aumslag--{S}olitar groups}.
  \newblock PhD thesis, D\"usseldorf, 2024.
  
\bibitem{Pyb03} L.~Pyber.
  \newblock Old groups can learn new tricks.
  \newblock In {\em Groups, combinatorics \& geometry ({{D}}urham,
    2001)}, pages 243--255. World Sci. Publ., River Edge, NJ, 2003.

% \bibitem{RZ10} L.~Ribes and P.~Zalesskii.
%   \newblock {\em Profinite groups}.
%   \newblock Springer Berlin Heidelberg, Berlin, Heidelberg, 2010.

\bibitem{Seg01} D.~Segal.
\newblock The finite images of finitely generated groups.
\newblock {\em Proc.\ Lond.\ Math.\ Society}, 82(3):597--613, 2001.

\bibitem{Spr98} T.~A. Springer.
  \newblock {\em Linear algebraic groups}.
  \newblock Birkh{\"a}user Boston, Boston, MA, second edition, 1998.

\bibitem{St60} R.\ Steinberg.
  \newblock Automorphisms of finite linear groups.
  \newblock {\em Canadian J.\ Math.}, 12:606--615, 1960.

\bibitem{St68} R.\ Steinberg.
  \newblock {\em Lectures on Chevalley groups}.
  \newblock Yale University, New Haven, CT, 1968.
  
\bibitem{Ste89} J.~R. Stembridge.
  \newblock Shifted tableaux and the projective representations of
  symmetric groups.
  \newblock {\em Adv.\ Math.}, 74(1):87--134, 1989.

\bibitem{Wie80} J.~Wiegold.
\newblock Growth sequences of finite groups. {{IV}}.
\newblock {\em J.\ Austral.\ Math.\ Soc.}, 29(1):14--16, 1980.

\bibitem{Wil98} J.~S.~Wilson.
  \newblock {\em Profinite groups}, London Math.\ Soc.\ Monogr.\ (N.S.), 19.
  \newblock The Clarendon Press, Oxford University Press, New York,
  1998.

% \bibitem{Wil09} R.~A. Wilson.
%   \newblock {\em The finite simple groups}, volume 251 of {\em
%     Graduate {{Texts}} in {{Mathematics}}}.
% \newblock {Springer London}, {London}, 2009.
  
\end{thebibliography}
%\bibliographystyle{abbrv}

%%%%%

\end{document}